\documentclass[12pt, a4paper]{amsart}

\usepackage{graphicx}
\usepackage{wrapfig}
\usepackage{comment}
\usepackage{lipsum}

\usepackage{pifont}

\usepackage[
unicode=true,
plainpages = false,
pdfpagelabels,
bookmarks=true,
bookmarksnumbered=true,
bookmarksopen=true,
breaklinks=true,
backref=false,
colorlinks=true,
linkcolor = DarkBlue,
urlcolor  = DarkBlue,
citecolor = DarkRed,
anchorcolor = green,
hyperindex = true,
hyperfigures
]{hyperref}
\hypersetup{
pdftitle={An explicit embedding of the group Q into a finitely presented group},
pdfauthor={V.H. Mikaelian},
pdfsubject={An explicit embedding of the group Q into a finitely presented group}
}

\usepackage[dvipsnames]{xcolor}
\colorlet{LightRubineRed}{RubineRed!70!}
\colorlet{Mycolor1}{green!10!orange!90!}
\definecolor{DarkRed}{HTML}{cc0000}
\definecolor{ChapterHeadColor}{HTML}{cc0000}
\definecolor{PartHeadColor}{HTML}{cc0000}
\definecolor{DarkBlue}{HTML}{0000cc}
\definecolor{QuoteColor}{HTML}{665665}

\usepackage[top=2.4cm, bottom=2.4cm, outer=2.5cm, inner=2.5cm, heightrounded
]{geometry}

\newcommand{\B}{{\mathcal B}}
\renewcommand{\S}{{\mathcal S}}
\newcommand{\Zz}{{\mathcal Z}}
\newcommand{\E}{{\mathcal E}}

\newcommand{\X}{{\mathfrak X}}
\newcommand{\Y}{{\mathfrak Y}}

\newcommand{\BB}{\boldsymbol{\mathscr{S}}}

\newcommand{\Z}{{\mathbb Z}}
\newcommand{\Q}{{\mathbb Q}}

\newcommand{\x}{{\mathbf x}}
\newcommand{\y}{{\mathbf y}}

\newcommand{\1}{\{1\}}

\newcommand{\bigast}{\textrm{\footnotesize\ding{91}}}

\usepackage[bitstream-charter]{mathdesign}

\DeclareSymbolFont{cmsymbols}{OMS}{cmsy}{m}{n}
\SetSymbolFont{cmsymbols}{bold}{OMS}{cmsy}{b}{n}
\DeclareSymbolFontAlphabet{\mathcal}{cmsymbols}

\theoremstyle{plain}
\newtheorem{Theorem}{Theorem}[section]
\newtheorem{Lemma}[Theorem]{Lemma}
  
\newtheorem{Corollary}[Theorem]{Corollary}

\theoremstyle{definition}

\theoremstyle{remark}
\newtheorem{Remark}[Theorem]{Remark} 
 
\newtheorem{Agreement}[Theorem]{Agreement}

\numberwithin{equation}{section}

\usepackage[bitstream-charter]{mathdesign}

\DeclareSymbolFont{cmsymbols}{OMS}{cmsy}{m}{n}
\SetSymbolFont{cmsymbols}{bold}{OMS}{cmsy}{b}{n}
\DeclareSymbolFontAlphabet{\mathcal}{cmsymbols}

\begin{document}



\title[On explicit embeddings of $\Q$]{\large On explicit embeddings of $\Q$ into finitely\\ presented groups}

\author{V. H. Mikaelian
}

\begin{abstract}
Explicit embeddings of the group $\Q$ into a finitely presented group $\mathcal{Q}$ and into a $2$-generator finitely presented group 
$T_{\!\mathcal{Q}}$ are suggested. 
The constructed embeddings reflect questions mentioned by Johnson, 
Bridson, de la Harpe in the literature in late 1990s about possibility of such embeddings for $\Q$. 
Technique used  here is based on the methods with integer-valued sequences sets used by Higman, and with specific free constructions of groups, including free product with amalgamation, HNN-extension, an auxiliary technical structure of $\bigast$-construction, etc.  
\end{abstract}

\date{\today}

\maketitle

\setcounter{tocdepth}{1}

\let\oldtocsection=\tocsection
\let\oldtocsubsection=\tocsubsection
\let\oldtocsubsubsection=\tocsubsubsection
\renewcommand{\tocsection}[2]{\hspace{-12pt}\oldtocsection{#1}{#2}}
\renewcommand{\tocsubsection}[2]{\footnotesize \hspace{6pt} \oldtocsubsection{#1}{#2}}
\renewcommand{\tocsubsubsection}[2]{ \hspace{42pt}\oldtocsubsubsection{#1}{#2}}

{\footnotesize \tableofcontents}

\section{Introduction}

\noindent
Our objective is to construct explicit embeddings of the additive group of rational numbers $\Q$ into certain finitely presented groups.
This subject belongs to the context of the celebrated Higman Embedding Theorem \cite{Higman Subgroups of fP groups} stating that a finitely generated group 
can be embedded into a finitely presented group if and only if it is recursive.  Under a \textit{recursive} group we understand a group 
with at most countable (effectively enumerable) generators and with a recursively enumerable set of defining relations.
A newer term for recursive is \textit{computable}, but we are going to stick to the traditional notation here. We namely prove:

\begin{Theorem}
\label{TH embedding of Q into GP group}
There is an explicit embedding of the additive group of rational numbers $\Q$ into a finitely presented group explicitly given by its generators and defining relations.
\end{Theorem}

We are going to build two instances of the finitely presented group promised in Theorem~\ref{TH embedding of Q into GP group}. The first is the group $\mathcal Q$ given by $98$ generators and a large number of defining relations listed in point~\ref{SU Writing the explicit  list of generators and relations for Q}.
The explicit embedding $\varphi: \Q \to \mathcal{Q}$ is given in \eqref{EQ define varphi} as the composition of $\alpha$ from \eqref{EQ define alpha} and of $\beta$ from \eqref{EQ define beta}.
\;
And the second instance is the group $T_{\!\mathcal{Q}}$ with only $2$ generators and again a large number of defining relations, see point~\ref{SU Writing the explicit  list of generators and relations for T cal Q} and the presentation \eqref{EQ T cal Q}.
The explicit embedding $\psi:\Q \to T_{\!\mathcal{Q}}$ is the composition of the above $\varphi$ and of $\gamma$ from \eqref{EQ define gamma}. In fact, $T_{\!\mathcal{Q}}$ contains the first group $\mathcal{Q}$ also.

\subsection{The problem on explicit embedding of $\Q$}
\label{SU The problem on explicit embedding of Q} 
As the group $\Q$ certainly is recursive, see point~\ref{SU Embedding Q into a 2-generator group},  
it is embeddable into a  finitely presented group by \cite{Higman Subgroups of fP groups}. 
The task of finding an \textit{explicit} embedding of that type for $\Q$ goes back to Graham Higman himself. 

Discussing the Higman Embedding
Theorem in \cite{Johnson on Higman's interest} Johnson 
displays explicit embeddings of some recursive groups into finitely presented ones. He expresses his gratitude to  Higman for raising that problem, and then concludes by mentioning: \textit{``Our main aim, of embedding in a finitely presented group the additive group of rational numbers continues to elude us''}.

Bridson and de la Harpe ask in Problem 14.10 of 14'th edition of Kourovka Notebook \cite{kourovka} in 1999: \textit{``(Well-known problem). It is known that any recursively presented group embeds in a
finitely presented group \cite{Higman Subgroups of fP groups}. Find an explicit and \textit{``natural''} finitely presented group $\Gamma$
and an embedding of the additive group of the rationals $\Q$ in $\Gamma$''}.
In the current edition of Kourovka Notebook this question is in Problem 14.10 (a).

De la Harpe on page 53 in \cite{De La Harpe 2000} stresses:
\textit{``We can also record the well-known problem of finding a natural and explicit embedding of $\Q$ in a finitely-presented group. Such a
group exists, by a theorem of Higman''}.

\smallskip
Problem 14.10 in \cite{kourovka} had one more point, which in the current edition  is labeled as Problem 14.10 (b): 
\textit{``Find an explicit embedding of $\Q$ in a finitely generated group; such a group exists by Theorem IV in
\cite{HigmanNeumannNeumann}''} (this is the weaker version of the previous question with condition of finite \textit{presentation} for $\Gamma$ dropped). To answer this question we in  \cite{On a problem on explicit embeddings of Q} were able to build $2$-generated groups explicitly containing  $\Q$  by two methods: using wreath products, and using free constructions of groups. 
However, our main aim of embedding of $\Q$ into a finitely \textit{presented} group continued to elude us also.

\subsection{Results of recent years}
\label{SU Results of recent years} 

Our attempts to find an embedding of a given \textit{general} recursive group $G$ (such as $\Q$) into a finitely presented one were based on modifications of the steps of Higman's construction in \cite{Higman Subgroups of fP groups}.
In particular, the first step of Higman's embedding is the construction of an embedding of $G$ into a $2$-generator group (let us denote it by $T_G$), the defining relations of which also are recursively enumerable. We suggested a universal method that automates this step, and also conveys certain properties (useful for embeddings into finitely presented groups) from  $G$ to $T_G$, see \cite{Embeddings using universal words} for details. 

In \cite{Explicit embeddings Moscow 2018} we were able to report a general method of explicit embeddings for some types of recursive groups, including $\Q$, into finitely presented groups. 
\cite{The Higman operations and  embeddings} contains an algorithm of how the sets of integer-valued sequences, used in Higman embeddings, can explicitly be written via an $H$-machine method (for $\Q$ see Example 3.5 and Remark 3.6 in \cite{The Higman operations and  embeddings}). 
And in \cite{A modified proof for Higman} we suggested a modification of the original Higman embedding  which not only is, we hope, simpler than \cite{Higman Subgroups of fP groups}, but which also makes explicit embeddings manageable. 
%
However, \cite{Explicit embeddings Moscow 2018, The Higman operations and  embeddings, A modified proof for Higman} do \textit{not} suggest any explicit finitely presented overgroup of $\Q$.

\smallskip
Interesting explicit examples of finitely presented groups holding $\Q$ were presented by 
Belk, Hyde and Matucci in \cite{Belk Hyde Matucci}.  
The first embedding of $\Q$ in \cite{Belk Hyde Matucci} is into the group $\overline T$ from \cite{Ghys Sergiescu}, namely, $\overline T$ is the group of all piecewise-linear homeomorphisms $f:\mathbb R \to \mathbb R$ satisfying certain specific requirements \cite[Theorem 1]{Belk Hyde Matucci}. 
The second finitely presented group of \cite{Belk Hyde Matucci} is related to the first one, and it is the automorphism group of Thompson's group $F$ \cite[Theorem 2]{Belk Hyde Matucci}. These embedding allow further variations, say, $\overline T$ (together with its subgroup $\Q$) admits an embedding into two specific finitely presented \textit{simple} groups $T\! \mathcal A$ and $T\mathcal V$, see \cite[Remark 4]{Belk Hyde Matucci}. Also, \cite[Remark 5]{Belk Hyde Matucci} refers to 
\cite{Hurley, Ould Houcine} where existence of a
finitely presented group with its center isomorphic to $\Q$ was proved; it would be interesting to find a natural example of such groups.
By \cite[Proposition 1.10]{Belk Hyde Matucci} 
the group $\overline T$ can be given by just $2$ generators and $4$ defining relations.
The results of \cite{Belk Hyde Matucci} were also used in  \cite{Belk Bleak Matucci Zaremsky}.

\smallskip
The recent progress motivated us to continue  
\cite{Embeddings using universal words, The Higman operations and  embeddings} to build explicit embeddings of $\Q$ into two finitely presented groups, denoted via $\mathcal{Q}$ and $T_{\!\mathcal{Q}}$ in Section~\ref{SE Explicit list of generators and relations for Q for TQ}. 
%

\subsection{Structure of this note}
\label{SU Structure of this}

Since we wish to avoid repetitions of  any parts from our other articles, here we just put very quick references to proofs, definitions, notations discussed elsewhere.
In particular, the Higman operations on integer-valued sequences, basic properties of benign subgroups, properties of free constructions and of the  auxiliary $\bigast$-construction (which we suggest for technical purposes) are not being defined here, and we just refer to \cite{Embeddings using universal words, A modified proof for Higman, Auxiliary free constructions for explicit embeddings} for details.  Section~\ref{SE Preliminary notation, constructions and references} below holds some quick definitions and exact references of that type.

In Section~\ref{SU Writing Higman code for the embedding} we use the method of \cite{Embeddings using universal words} to build an embedding $\alpha$ of $\Q$ into a $2$-generator group $T_\Q$ the defining relations of which can explicitly be written as \eqref{EQ TG genetic code}. 
Then using \cite{The Higman operations and  embeddings} we code the relations of \eqref{EQ TG genetic code} via integer-valued sequences $f_k$ in \eqref{EQ Higman code for Q}. The set of such sequences is denoted by $\mathcal T$.

In the free group $F=\langle a,b,c \rangle$ of rank $3$ we using this set $\mathcal T$ define a specific subgroup $A_{\mathcal T} = \langle a_{f_k} \mathrel{|} f_k\in \mathcal T \, \rangle$. 
The heavy part of the current work is to show that $A_{\mathcal T}$ is benign in $F$, and this job is accomplished in sections \ref{SE Construction of KB}\,--\ref{SE Construction of KF and KT}.

Section~\ref{SE The final embedding} uses the benign subgroup $A_{\mathcal T}$ and the respective groups $K_{\mathcal T}, L_{\mathcal T}$  to build the finitely presented overgroup $\mathcal{Q}$ of $T_\Q$, together with the embedding $\beta: T_\Q \to \mathcal{Q}$. Hence, the desired explicit embedding $\varphi: \Q \to \mathcal{Q}$ is the composition of $\alpha$ and $\beta$. 

In Section~\ref{SE Explicit list of generators and relations for Q for TQ} we  explicitly write down  $\mathcal{Q}$ by its generators and defining relations, see   point~\ref{SU Writing the explicit  list of generators and relations for Q}. 
Further, we again use the method of \cite{Embeddings using universal words} to embed $\mathcal{Q}$ (and hence its subgroup $\Q$ also) into a $2$-generator group $T_{\!\mathcal{Q}}$ with a still large number of relations, see Remark~\ref{RE It is possible to reduce the above number of relations}. The generators and defining relations of $T_{\!\mathcal{Q}}$ are outlined in point~\ref{SU Writing the explicit  list of generators and relations for T cal Q}, and the explicit embedding $\psi: \Q \to T_{\!\mathcal{Q}}$ is the composition of the above $\varphi$ with $\gamma$ from \eqref{EQ define gamma}.
In order to have a short name to label the algorithm of sections \ref{SU Writing Higman code for the embedding}--\ref{SE Explicit list of generators and relations for Q for TQ} we may call it an $H$-machine (running on $\Q$).

\subsection*{Acknowledgements}
\label{SU Acknowledgements}

The work is supported by the grant 21T-1A213 of SCS MES Armenia. 
I am thankful to the Fachliteratur Program of the German Academic Exchange Service DAAD for the academic literature provided over past years, grant A/97/13683.

\section{Preliminary notation, constructions and references}
\label{SE Preliminary notation, constructions and references}

\subsection{Integer-valued functions $f$}
\label{SU Integer functions f} 

Following \cite{Higman Subgroups of fP groups} denote by $\mathcal E$ the set of all functions $f : \Z \to \Z$ with finite supports. 
When for a certain  $m=1,2,\ldots$ we have 
$f(i)=0$ for \textit{all} $i<0$ and $i\ge m$, then it is comfortable to record $f$ as a sequence $f=(j_0,\ldots,j_{m-1})$ assuming $f(i)=j_i$ for $i=0,\ldots,m-1$, e.g., the function $f=(0,5,9,8)$ sending the integers $0, 1, 2, 3$ to $0,5,9,8$ (and all other integers to $0$).
Denote the set of all such functions written as sequences for a fixed $m$ by $\E_m$.\, 
Clearly, $m$ may not be uniquely defined for a given $f$\!, and where necessary we may append extra zeros at the end of a sequence, e.g., the previous function can well be recorded as $f=(0,5,9,8,0,0)\in \E_6$ (the last two zeros change nothing in the way $f$ acts on $\Z$). Also, the constant zero function can be written as $f=(0)$ or, say, as $f=(0,0,0)$ where necessary. 
See more on such functions in  point 2.2 in \cite{The Higman operations and  embeddings}.

For any $f\in \mathcal E$ and $k\in \Z$  define the function $f_{k}^+$\! as follows: 
$f_{k}^+(i)=f(i)$ for all $i\!\neq\! k$, and   
$f_{k}^+(k) = f(k)\!+\!1$.
When $f\in \E_m$ (with $m$ given by the context), we shorten $f_{m-1}^+=f^+$\!. Say, for the above $f\!=\!(0,5,9,8)\in\E_4$ we have $f_{1}^+\!\!=(0,6,9,8)$ and 
$f^+\!\!=(0,5,9,9)$, i.e., we just add $1$ to the last  coordinate of $f$ to get $f^+$\!.

\subsection{The Higman operations}
\label{SU The Higman operations} 
Higman defines special operations which  transform subsets of $\E$ to some new subsets of  $\E$, see Section 2 in In \cite{Higman Subgroups of fP groups}.
To make applications of Higman operations more effective we defined some extra auxiliary operations, see points 2.3 and 2.4 in  \cite{The Higman operations and  embeddings}.
In this note we use the Higman operations to compose such a specific subset $\mathcal T$ of $\E$ which in some sense ``records'' the defining relations of $\Q$. 
Although Higman operations will be used in this note often, we do not want to make the below proofs dependant on understanding of Higman operations. Hence we have an:

\begin{Agreement}
\label{AG Agreement about Higman operations}
Each time we use a subset of $\E$ built by Higman operations, we will also give an alternative description of it. 
Say, in Section~\ref{SE Construction of KB} we are going to use the subset 
${\mathcal B}=\upsilon(\zeta_{\!1} \Zz , \tau\S)$ of $\E_2$, but the reader does \textit{not} have to learn what the Higman operations $\upsilon, \zeta_{\!1}, \tau$, and the symbols $\Zz, \S$ mean, as it will be at once explained that ${\mathcal B}$ is just the set of all functions of type either $(0,n)$ or $(n\!+\!1,\, n)$, $n\in \Z$.
\end{Agreement}

\subsection{Benign subgroups}
\label{SU Benign subgroups and Higman operations} 

Higman calls a subgroup $H$ of a finitely generated group $G$ a \textit{benign} subgroup, if $G$ can be embedded into a finitely presented group $K$ with a finitely generated subgroup $L$ such that $G \cap L = H$. If we wish to stress the correlation of $K$ and $L$ with $G$, we may denote them $K_G$ and $L_G$.
For detailed information on benign subgroups we refer to  Sections 3, 4 in \cite{Higman Subgroups of fP groups}, see also Section 3 in \cite{A modified proof for Higman}.

\begin{Remark}
\label{RE finite generated is benign}
From definition of benign subgroups it is very easy to see that arbitrary finitely generated subgroup $H$ in any finitely presented group $G$ is benign in $G$. We are going to often use this remark in the sequel.
\end{Remark}

\subsection{Free constructions}
\label{SU Free constructions} 

For background information on free products with amalgamation and on HNN-extensions we refer to 
\cite{Bogopolski} and \cite{Lyndon Schupp}. 
Our usage of the \textit{normal forms} in
free constructions is closer to~\cite{Bogopolski}.
Notations vary  in the literature, and to maintain uniformity we are going to adopt notations of \cite{A modified proof for Higman}. 

If any groups $G$ and $H$ have subgroups, respectively, $A$ and $B$ isomorphic under  $\varphi : A \to B$, then the (generalized) free product of $G$ and $H$ with amalgamated subgroups $A$ and $B$ is denoted by
$G*_{\varphi} H$ (we are not going to use the alternative notation $G*_{A=B} H$). When $G$ and $H$ are overgroups of the same subgroup $A$, and $\varphi$ is just the identical  map on $A$, we write the above as $G*_{A} H$.

If $G$ has subgroups $A$ and $B$ isomorphic under  $\varphi : A \to B$, then the HNN-extension  of the base $G$ 
by some stable letter $t$
with respect to the isomorphism 
$\varphi$ is denoted by
$G*_{\varphi} t$.
In case when $A=B$ and $\varphi$ is  identical  map on $A$, denote the  above by $G*_{A} t$.
We also use HNN-extensions $G *_{\varphi_1, \varphi_2, \ldots} (t_1, t_2, \ldots)$  
with more than one stable letters, see \cite{A modified proof for Higman} for details.

\medskip
Below we are going to use a series of facts about certain specific subgroups in  free constructions $G*_{\varphi} H$, \, $G*_{A} H$, \,$G*_{\varphi} t$ and $G*_{A} t$. We have stockpiled them in Section 3 of  \cite{Auxiliary free constructions for explicit embeddings} to refer to that section whenever needed.

\subsection{The $\bigast$-construction}
\label{SU The *-construction}

Although the constructions in \cite{Higman Subgroups of fP groups} look very diverse, they often seem to be \textit{particular cases} of a single general construction. Thus, for our work it is reasonable to define that general construction and to collect its basic properties in  \cite{Auxiliary free constructions for explicit embeddings} in order to refer to them where needed. 

\smallskip
Let $G\le M  \le K_1,\ldots,K_r$ be an arbitrary system of groups  such that $K_i \cap  K_j=M$ for any distinct indices $i,j=1,\ldots,r$.
Choosing in each $K_i$ a subgroup $L_i$ we can build an auxiliary ``nested'' free construction:
\begin{equation}
\label{EQ initial form of star construction}
\Big(\cdots
\Big( \big( (K_1 *_{L_1} t_1) *_M (K_2 *_{L_2} t_2) \big) *_M  (K_3 *_{L_3} t_3)\Big)\cdots 
\Big) *_M  (K_r *_{L_r} t_r)\,.
\end{equation}
To avoid this very bulky notation, let us for the sake of briefness denote the above  by:
\begin{equation}
\label{EQ nested Theta  multi-dimensional short form} 
\textstyle
\bigast_{i=1}^{r}(K_i, L_i, t_i)_M.
\end{equation}
Denote $G \cap \, L_i = A_i$, $i=1,\ldots,r$. If for each $i$ we are limited to $K_i=G$, $L_i=A_i$, $M=G$, then 
$\bigast_{i=1}^{r}(G, A_i, t_i)_G$ is noting but the usual HNN-extension $G *_{A_1,\ldots,\,A_r } \!(t_1,\ldots,t_r)$.

\medskip
All the facts in the rest of the current point are proved in  \cite{Auxiliary free constructions for explicit embeddings}.

\begin{Lemma}
\label{LE intersection in bigger group multi-dimensional}
If $G\le M  \le K_1,\ldots,K_r$ are groups mentioned above, then in ${\bigast}_{i=1}^{r}(K_i, L_i, t_i)_M$
the following equality holds:
$$
\langle G, t_1,\ldots,t_r \rangle= 
G *_{A_1,\ldots,\,A_r } \!(t_1,\ldots,t_r).
$$
\end{Lemma} 

For any group $G$ and its subgroup $A$ 
the well known equality $G \cap G^t \!= A$ in the HNN-extension $G*_A t$  can be generalized to the following:

\begin{Lemma}
\label{LE intersection in HNN extension multi-dimensional}
Let $A_1,\ldots,\,A_r$ be any subgroups in a group $G$ with the intersection 
$I=\bigcap_{\,i=1}^{\,r} \,A_i$. 
Then in $G *_{A_1,\ldots,\,A_r} (t_1,\ldots,t_r)$ we have:
\begin{equation}
\label{EQ gemeral intersection in HNN}
\textstyle
G \cap G^{t_1 \cdots\, t_r}
= I.
\end{equation}
\end{Lemma}

\begin{Lemma}
\label{LE join in HNN extension multi-dimensional}
Let $A_1,\ldots,\,A_r$ be any subgroups in a group $G$ with the join
$J=\big\langle\bigcup_{\,i=1}^{\,r} \,A_i\big\rangle$. 
Then in $G *_{A_1,\ldots,\,A_r} (t_1,\ldots,t_r)$ we have:
\begin{equation}
\label{EQ gemeral intersection in HNN multi-dimensional}
\textstyle 
G \cap \big\langle 
  \bigcup_{\,i=1}^{\,r} \,G^{t_i} \big\rangle
=J.
\end{equation}
\end{Lemma}

Now the $\bigast$-construction of \eqref{EQ nested Theta  multi-dimensional short form} together with 
Lemma~\ref{LE intersection in bigger group multi-dimensional},
Lemma~\ref{LE intersection in HNN extension multi-dimensional} and 
Lemma~\ref{LE join in HNN extension multi-dimensional}
provide a corollary which will be extensively used below.

\begin{Corollary}
\label{CO intersection and join are benign multi-dimensional}
If the subgroups $A_1,\ldots,\,A_r$ are benign in a finitely generated group $G$, then:
\begin{enumerate}
\item 
\label{PO 1 CO intersection and join are benign multi-dimensional}
their intersection $I=\bigcap_{\,i=1}^{\,r} \,A_i$ also is benign in $G$;
\item 
\label{PO 2 CO intersection and join are benign multi-dimensional}
their join $J=\big\langle\bigcup_{\,i=1}^{\,r} \,A_i\big\rangle$ also is benign in $G$.
\end{enumerate}
Moreover, if the finitely presented groups $K_i$ with their finitely generated subgroups $L_i$ can be 
given for each $A_i$ explicitly, then the respective finitely presented overgroups $K_I$ and $K_J$ with finitely generated  subgroups $L_I$ and $L_J$
can also be given for $I$ and for $J$ explicitly.
\end{Corollary}

The following corollary lets us  discover free products inside $G  *_{A_1,\ldots,\,A_r } \!(t_1,\ldots,t_r)$, and hence inside
$\bigast_{i=1}^{r}(K_i, L_i, t_i)_M$, as soon as some \textit{smaller} free products are found inside $G$:

\begin{Corollary}
\label{CO smaller free product to the larger free product}
Let $A_1,\ldots,\,A_r$ be any subgroups in a group $G$ such that their join $J$ in $G$ is isomorphic to their free product
$\prod_{i=1}^{r} \,A_i$.  
Then the join
$\big\langle 
\bigcup_{\,i=1}^{\,r} G^{t_i} \big\rangle$
is isomorphic to the free product $\prod_{i=1}^{r} G^{t_i}$
in $G *_{A_1,\ldots,\,A_r } \!(t_1,\ldots,t_r)$, and hence in $\bigast_{i=1}^{r}(K_i, L_i, t_i)_M$.
\end{Corollary}

\subsection{Auxiliary benign free subgroups of countable rank}
\label{SU Auxiliary benign free subgroups}
In the free group $\langle b,c \rangle$ of rank $2$ denote $b_i = b^{c^i}=c^{-i} b c^i$ for any $i\in \Z$.
In \cite{Auxiliary free constructions for explicit embeddings} for a fixed integer
$m$ we have used the isomorphisms
$\xi_m(b)=b_{-m+1},\;
\xi'_m(b)=b_{-m}$,\;
$\xi_m(c)=\xi'_m(c)=c^2$
in the free group $\langle b,c \rangle$ of rank $2$,  and have defined the HNN-extension 
$
\Xi_m = \langle b,c \rangle *_{\xi_m, \xi'_m} (t_m, t'_m)
$
with stable letters $t_m, t'_m$. In \cite{Auxiliary free constructions for explicit embeddings} we have proved the following technical facts:
\begin{Lemma}
\label{LE Ksi}
In the above notation the following equalities hold for any $m$ in $\Xi_m$:
\begin{equation}
\label{EQ Ksi two equality}
\begin{split}
\langle b,c \rangle \cap \langle b_m, t_m, t'_m\rangle &= \langle b_m, b_{m+1},\ldots\rangle,
\\
\langle b,c \rangle \cap \langle b_{m-1}, t_m, t'_m\rangle &= \langle b_{m-1}, b_{m-2},\ldots\rangle.
\end{split}
\end{equation}
\end{Lemma}

This means that the subgroup of countable rank $\langle b_m, b_{m+1},\ldots\rangle$ is benign in  $\langle b,c \rangle$, and as the finitely presented overgroup and as its finitely generated subgroup we may take $\Xi_m$ and $\langle b_m, t_m, t'_m\rangle$. Similarly, $\langle b_{m-1}, b_{m-2},\ldots\rangle$ is benign for the same finitely presented overgroup $\Xi_m$ and its finitely generated subgroup $\langle b_{m-1}, t_m, t'_m\rangle$.

\medskip
For simplicity of notations below denote $F=\langle a,b,c \rangle$ the free group of rank $3$. Take the ordinary free product $\langle a \rangle * \,\Xi_m$ which clearly is finitely presented.

\begin{Lemma}
\label{LE Ksi for G}
In the above notation the following equalities hold for any $m$ in $\langle a \rangle * \,\Xi_m$:
\begin{equation}
\label{EQ a simple example of bening subgroup}
\begin{split}
F \cap \langle b_m, t_m, t'_m\rangle = \langle b_m, b_{m+1},\ldots\rangle
\;\;\; {\it and} \;\;\;  
F \cap \langle a, b_m, t_m, t'_m\rangle = \langle a, b_m, b_{m+1},\ldots\rangle,
\hskip3mm \\
F \! \cap \!\langle b_{m-1}, t_m, t'_m\rangle  \! =  \! \langle b_{m-1}, b_{m-2},\ldots\rangle
\;\; {\it and} \;\;\,  
F \! \cap \! \langle a, b_{m-1}, t_m, t'_m\rangle \!  =  \! \langle a, b_{m-1}, b_{m-2},\ldots\rangle.
\end{split}
\end{equation}
\end{Lemma}

This means that in $F$ a subgroup of each of the  types \;
$\langle b_m, b_{m+1},\ldots\rangle$,\;
$\langle a, b_m, b_{m+1},\ldots\rangle$,\;
$\langle b_{m-1}, b_{m-2},\ldots\rangle$,\;
$\langle a, b_{m-1}, b_{m-2},\ldots\rangle$ is benign.  For each of them as a finitely presented overgroup of $F$ we may choose the above group $\langle a \rangle * \,\Xi_m$, 
and as the finitely generated subgroup of the latter we can respectively choose the subgroups \;$\langle b_m, t_m, t'_m\rangle$,\;
$\langle a, b_m, t_m, t'_m\rangle$,\;\;
$\langle b_{m-1}, t_m, t'_m\rangle$,\;\;
$\langle a, b_{m-1}, t_m, t'_m\rangle$. 

\medskip
Further, we are able to get new benign subgroups by ``attaching'' subgroups of types discussed in Lemma~\ref{LE Ksi} and Lemma~\ref{LE Ksi for G}. For example, for two finitely presented groups
$\Xi_m *_{\langle b_m, t_m, t'_m\rangle} u$ and
$\Xi_0 *_{\langle b_{-1}, t_0, t'_0\rangle} v$ build the finitely presented
$\bigast$-construction:
\begin{equation}
\label{EQ benign sample ONE OTHER}
\left(\Xi_m *_{\langle b_m, t_m, t'_m\rangle} u \right)
\, *_{\langle b,c\rangle}
\left(\Xi_0 *_{\langle b_{-1}, t_0, t'_0\rangle} v \right).
\end{equation}
By
Corollary~\ref{CO intersection and join are benign multi-dimensional}\;\eqref{PO 2 CO intersection and join are benign multi-dimensional} the join\, 
$J=\langle \ldots b_{-2}, b_{-1};\;
b_m, b_{m+1},\ldots\rangle$
is benign in $\langle b,c \rangle$ for an arbitrary $m=1,2,\ldots$ As the respective finitely presented overgroup $K_J$ of  $\langle b,c \rangle$ we may take \eqref{EQ benign sample ONE OTHER}, and as the finitely generated subgroup $L_J$ of the latter we may pick its $4$-generator subgroup $\big\langle 
\langle b,c\rangle^u, \langle b,c\rangle^v
\big\rangle$; see Lemma~\ref{LE join in HNN extension multi-dimensional} and the details in \cite{Auxiliary free constructions for explicit embeddings} for verification of:
\begin{equation}
\label{EQ auxiliary b-2 b-1 a b_m, b_m+1}
F \cap
\big\langle 
\langle b,c\rangle^u, \langle b,c\rangle^v
\big\rangle 
=
\langle \ldots b_{-2}, b_{-1};\;
b_m, b_{m+1},\ldots\rangle.
\end{equation}

\subsection{The ``conjugates collecting'' process}
\label{SU The conjugates collecting process}

Let $\X$ and $\Y$ 
be some disjoint subsets in any group $G$. Then any element $w\in \langle \X,\Y \rangle$ can be written as:
\begin{equation*}
\label{EQ elements <X,Y>}
w=u\cdot v
= x_1^{\pm v_1}
x_2^{\pm v_2}
\cdots
x_{k}^{\pm v_k}
\cdot
v
\end{equation*}
with some $v_1,v_2,\ldots,v_k,\, v\in\langle \Y \rangle$, and 
$x_1,x_2,\ldots,x_k \in \X$. The proof, examples and variations of this fact can be found in  
point~2.6 in \cite{A modified proof for Higman}.

In particular, for $\X=\{x\}$ and $\Y=\{y\}$ in a 2-generator group $G=\langle x, y\rangle$ we may present any element $w \in G$ as a product
of some conjugates of $x$ and of some power of $y$:
\begin{equation}
\label{EQ elements from <x,y>}
w=x^{\pm y^{n_1}}\!x^{\pm y^{n_2}}\!\!\cdots\, x^{\pm y^{ n_s}}\!\!\cdot y^k=\!u\cdot v\,.
\end{equation}


\section{Writing Higman code for the embedding of $\Q$}
\label{SU Writing Higman code for the embedding}

\subsection{Embedding $\Q$ into a $2$-generator group $T_\Q$}
\label{SU Embedding Q into a 2-generator group} 

The group $\Q$ has recursive presentation:
\begin{equation}
\label{EQ Q defining relations}
\Q = \big\langle x_1, x_2,\ldots \mathrel{|} 
x_{k}^{k}=x_{k-1},\;\, k\!=\!2,3\ldots \big\rangle
\end{equation}
where the generators $x_k$ corresponds to  fractional numbers ${1 \over k!}$ with $k=2,3\ldots$, see p.~70 in \cite{Johnson}.
To bring this to a format more appropriate for our embedding rewrite each $x_k^k=x_{k-1}$ as 
$x_k^k\,x_{k-1}^{-1}=1$, and use the latter as the $k$'th relation $r_k$ of $\Q$ for  $k=2,3\ldots$ below.

A method automatically outputting the embedding of a countable group into a $2$-generator group was given earlier, see 
Introduction,
formula (1.2),
Theorem 1.1, and more specifically,
Point 3.3,
formula (3.6),
and Theorem 3.2 in \cite{Embeddings using universal words}.
In particular as $\Q$ is torsion-free, we can use the shorter formula (3.6) in \cite{Embeddings using universal words}
to map each $x_k= {1 \over k!}$ to the word: 
$$
\textstyle
\alpha\big({1 \over k!}\big)=y^{(x y^k)^{\,2} y^{\!-1}}
\in \langle x,y \rangle$$
on just two letters $x,y$.
Then the relation $x_k^k\,x_{k-1}^{-1}$ is being
simplified to:
$$
w_k (x,y)=
\big(y^{(x y^k)^{\,2} x^{\!-1}} \big)^k 
\big(y^{(x y^{k-1})^{\,2} x^{\!-1}}\big)^{-1}\!\!
=\,
(y^k)^{(x y^k)^{\,2} x^{\!-1}} 
y^{-(x y^{k-1})^{\,2} x^{\!-1}}\!\!,
$$
using which we get an embedding:
\begin{equation}
\label{EQ define alpha}
\alpha:\Q\to T_\Q
\end{equation}
into the $2$-generator recursively presented group:
\begin{equation}
\label{EQ TG genetic code}
T_\Q=\big\langle x,y 
\;\mathrel{|}\;
(y^k)^{(x y^k)^{\,2} x^{\!-1}} 
y^{-(x y^{k-1})^{\,2} x^{\!-1}}
\!\!,\;\; k=2,3,\ldots
\big\rangle.
\end{equation}
For example, the image of the rational number 
${98 \over 12}={49 \over 6}=49\cdot {1 \over 3!}$\, in $T_\Q$ is:
$$
\textstyle
\alpha\big({98 \over 12}\big)
\,=\,
\left(y^{(x y^3)^{\,2}\, x^{\!-1}}\right)^{49}
\!
=\big(y^{49}\big)^{(x y^3)^{\,2}\, x^{\!-1}}
\!\! \in T_\Q.
$$

\subsection{Getting the  Higman code $\mathcal T$ and the group $A_{\mathcal T}$ for $T_\Q$}
\label{SU Getting the  Higman code from TQ} 
For any group, given by  two generators $x,y$, Higman writes its defining relations\; $w=x^{j_0}y^{j_1}\cdots x^{j_{2r}}y^{j_{2r\!+\!1}} $ via sequences 
\begin{equation}
\label{EQ f written by j}
f=(j_0,\; j_1, \;\ldots\;, j_{2r},\; j_{2r\!+\!1})
=
\big(\,
f(0),\; f(1), \;\ldots\; , f(2r),\; f(2r\!+\!1)
\big)
\end{equation}
of integers by just recording the exponents of two generators, see p. 473 in \cite{Higman Subgroups of fP groups} (the cases $f(0)\!=\!0$, $f(2r\!+\!1)\!=\!0$ are not ruled out).
Since we intentionally prepared the $2$-generator group $T_\Q$ in the needed format \eqref{EQ TG genetic code}, it remains to write for each of its relations:
\begin{equation}
\label{EQ relations for Q detailed}
\begin{split}
w_k & = (y^k)^{(x y^k)^{\,2} x^{\!-1}} 
y^{-(x y^{k-1})^{\,2} x^{\!-1}}
\\
&=
x\,y^{-k} x^{-1} y^{-k} x^{-1} \cdot y^k \cdot x\, y^k x\;  y \cdot
x^{-1} y^{1-k} x^{-1} \cdot y^{-1} \cdot
x \, y^{k-1}\, x\, y^{k-1} x^{-1} 
\end{split}
\end{equation}
with $k=2,3,\ldots$, the respective sequence: 
\begin{equation}
\label{EQ Higman code for Q}
f_k=
\big(1,-k,-1, -k , -1,\; k,\;  1,\;  k, \; 1,\;  1, -1,\;  1\!\!-\!k, -1,-1,\;  1,\;  k\!-\!\!1,\;  1,\;  k\!-\!\!1,\;  -1\big).
\end{equation}
And for the set of the mentioned relations $w_k$ of $T_\Q$ we get the countable set $\mathcal T = \big\{ 
f_k \mathrel{|} k=2,3,\ldots \big\}$
of sequences of type \eqref{EQ Higman code for Q}.

\begin{Remark}
In Example 3.5 of \cite{The Higman operations and  embeddings} for some specific reasons the \textit{permuted} version of the above code was recorded as
$
\alpha\, f_{k} \!=\!
\big(6\times 1,\;\; \; 
6\times -1,\;\; 
2\times k,\;\; 
2\times -k,\;\; 
1\!-\!k,\;\; 
2\times (k\!-\!\!1)
\big)
$
for a suitable permutation $\alpha \in S_{19}$.
But in the current note we use \eqref{EQ Higman code for Q} instead.
\end{Remark}

In the free group $F=\langle a,b,c \rangle$ of rank $3$ for \textit{any} sequence \eqref{EQ f written by j}  from $\E$ denote 
$$
b_f = b_0^{f(0)} b_{1}^{f(1)} \cdots b_{2r}^{f(2r)} b_{2r\!+\!1}^{f(2r\!+\!1)} \in \, F
$$
where $b_i = b^{c^i}$ as in point~\ref{SU Auxiliary benign free subgroups}. Also denote $a_f = a^{b_f}$, and 
for a given set $\B$ of functions $f \in \E$
define the subgroup $A_{\B} = A_{\B}(a,b,c)=\langle a_f
\mathrel{|} f \in \B \rangle$ of $F$. In particular, when $\B$ is the set $\mathcal T$ consisting of sequences $f_k$ from \eqref{EQ Higman code for Q}, we in $F$ get the elements $b_{f_k}, a_{f_k}$, and the subgroup  $A_{\mathcal T} = \langle a_{f_k} \mathrel{|} f_k\in \mathcal T \, \rangle$.

The main objective of sections \ref{SE Construction of KB}\,--\ref{SE Construction of KF and KT} is to present a step-by-step construction mechanism to achieve the following two goals:
\begin{enumerate}
\item The set $\mathcal T$ can be constructed by Higman operations, 
\item The subgroup $A_{\mathcal T}$ is benign in $F$.
\end{enumerate}

These questions will be answered in point~\ref{SU Obtaining a larger set}, where for $A_{\mathcal T}$ we give the respective embeddings, and the groups $K_{\mathcal T}$ and $L_{\mathcal T}$  \textit{explicitly}.

\section{Construction of $K_{\mathcal B}$ 
}
\label{SE Construction of KB}

\subsection{Building the group $\mathscr{A}$}
\label{SU Building the group script A} 

Denote by $\mathcal B$ the set of all sequences $f$ from set $\mathcal E_2$ of type either $(0,n)$ or $(n\!+\!1,\, n)$ for all $n\in \Z$. In terms of the Higman operations this is the set $\upsilon(\zeta_{\!1} \Zz , \tau\S)$, see Agreement~\ref{AG Agreement about Higman operations}. 
The objective of this section is to prove that $A_{\mathcal B} = \langle a_f \mathrel{|} f\in \mathcal B \rangle$ is benign in $F$, and to explicitly write  the related auxiliary groups in point~\ref{SU Writing K_B by generators and defining relations}.

Turning back to the notation of point~\ref{SU Auxiliary benign free subgroups},
and taking $m\!=\!1$ in Lemma~\ref{LE Ksi for G} we see that the subgroup 
$\langle b_1, b_2,\ldots \rangle$ is benign in $F$ for the finitely presented overgroup $\Theta = \langle a \rangle * \Xi_1$, and for its finitely generated subgroup $\langle b_1, t_1, t'_1\rangle$.
Also the subgroup $\langle a, b_{0}, b_{-1},\ldots\rangle$ is benign in $F$ for the same finitely presented $\Theta$ and for its finitely generated subgroup $\langle a, b_{0}, t_1, t'_1\rangle$.
Use these groups to build the finitely presented
$\bigast$-construction:
$$
\mathscr{C} 
= \big(\Theta *_{\langle b_1, t_1, t'_1 \rangle} u_1 \big)
\,*_{\Theta}
\big(\Theta *_{\langle a, b_{0}, t_1, t'_1 \rangle} u_2 \big)
$$
and notice that in $\mathscr{C}$ the conjugates $F^{u_1}$ and $F^{u_2}$ of $F$ have a useful feature: 

\begin{Lemma}
\label{LE GuGv is free}
In the above notations $F^{u_1}$ and $F^{u_2}$ generate the free product $F^{u_1}* F^{u_2}$ inside $\mathscr{C}$.
\end{Lemma}

\begin{proof}
Applying Lemma~\ref{LE intersection in bigger group multi-dimensional} to $\mathscr{C}$, and then Lemma~\ref{LE Ksi for G}
we see that $\mathscr{C}$ contains: 
$$
F*_{F \,\cap\, \langle b_1, t_1, t'_1 \rangle, 
\;\;\; 
F\,\cap\,\langle a, b_{0}, t_1, t'_1 \rangle} 
(u_1,u_2)
\;=\;
F*_{\langle b_1, b_2,\ldots  \rangle, \;\; \langle a, b_{0}, b_{-1},\ldots\rangle} (u_1,u_2).
$$
But in $F$ the subgroups 
$\langle b_1, b_2,\ldots  \rangle$ and $\langle a, b_{0}, b_{-1},\ldots\rangle$
generate their \textit{free} product.
Hence, we can apply Corollary~\ref{CO smaller free product to the larger free product} to get that  $F^{u_1}$ and $F^{u_2}$ also generate their free product.
\end{proof}

This lemma lets us continue \textit{any} two isomorphisms defined on $F^{u_1}$ and on $F^{u_2}$ to an isomorphism on the subgroup  $F^{u_1}* F^{u_2}$ of $\mathscr{C}$.
As such take the trivial automorphism in $F^{\,u_1}$ and the conjugation by $b^{u_2}$ in $F^{\,u_2}$\!\!,\; then denote their common continuation in $F^{u_1}* F^{u_2}$ by $\omega$.
Inside $F$ this $\omega$ leaves the elements $ b_1, b_2,\ldots $ intact, and $\omega$ sends $a, b_{0}, b_{-1},\ldots$ to their conjugates $a^b\!,\, b_{0}^b,\, b_{-1}^b,\ldots$
Also, $\omega$ can be determined by its values on just \textit{six} conjugates $a^{u_1}, b^{u_1}, c^{u_1}, a^{u_2}, b^{u_2}, c^{u_2}$\!.

Lastly, denote by $\delta$ the isomorphism of $F$ sending $a, b, c$ to $a, b^c, c$. Then HNN-extension
$$
\mathscr{A} = \mathscr{C} \! *_{\omega, \delta}\!(d,e)
$$
is finitely presented because $\mathscr{C}$ is finitely presented, while $\omega, \delta$ can be determined by their values on just \textit{nine} elements mentioned above. 
In analogy with $b_i=b^{c^i}$ denote $d_i = d^{e^i}$ in $\mathscr{A}$.
The following lemma displays the main effect for the sake of which the group $\mathscr{A}$ was so constructed (recall the notations $f_{j}^+$ and $f_{j}^-$ from point~\ref{SU Integer functions f}):

\begin{Lemma}
\label{LE action of d_m on f} For any $f \in \mathcal E$ and any $j=0,1,\ldots$ 
we have 
$
a_f^{d_j} =\! a_{f_{j}^+}$ 
and
$
a_f^{\,d_j^{-1}}\!\! = a_{f_{j}^-}
$ in $\mathscr{A}$.
\end{Lemma}

Since this lemma was proved earlier via equations (4.1), (4.2) and Lemma~4.1 in \cite{A modified proof for Higman} 
with slightly difference notations,  we omit the formal proof and bring a simple example that fully explains the proof argument.
If  $f=(2,5,3)$ and $j=1$, then $b_f=b_0^{2}\,b_1^5\,b_2^{3}$, and 
\begin{equation*}
\begin{split}
a_f^{d_1} & =
\big(
b_2^{\!-3}b_1^{\!-5}b_0^{\!-2}
\; a\;
b_0^{2}b_1^{5}b_2^{3} \,
\big)^{d_1}
\!\! \\
& = 
b_2^{-3}\,
\big(b_1^{\!-1} b_1^{\!-5} b_1^{\vphantom8}\big)
\big(b_1^{\!-1} b_0^{\!-2} b_1^{\vphantom8}\big)
\,\, \big(b_1^{\!-1} a b_1^{\vphantom8}\big) \,
\big(b_1^{\!-1} b_0^{2}b_1^{\vphantom8}\big)\,
\big(b_1^{\!-1} b_1^{5}b_1^{\vphantom8}\big)
\,b_2^{3}
\end{split}
\end{equation*}
which after simple cancellations is equal to
$
b_2^{\!-3}b_1^{\!-\,6}  b_0^{\!-2}
\;a\;
b_0^{2}b_1^6b_2^{3}
= a_{f_{1}^+}
$
with the sequence
$f_{1}^+ \!= (2,\,\boldsymbol{5\!+\!1}\,,3)= (2,\boldsymbol{6},3)$.
And in a similar way if $j=2$, then  we have $a_f^{d_2}=a_{f_2^+}\!=a_{f^+}$ where 
$f_{2}^+ \!= f^+ \!\!= (2,5,\,\boldsymbol{3\!+\!1})= (2,5,\boldsymbol{4})$.

\subsection{$\zeta_1 \mathcal Z$ is benign in $F$}
\label{SU zeta_1 Z is benign} 
Denote $N_1=\langle
a, d_1
\rangle$,\;
$B_1=\langle
a_{(0,n)} \mathrel{|} n\in \Z
\rangle=A_{\zeta_1 \mathcal Z}$, 
and show that $F \cap N_1 = B_1$ in $\mathscr{A}$.
Applying Lemma~\ref{LE action of d_m on f} for $a$ and $d_1$ repeatedly we 
get:
\begin{equation}
\label{EQ action of d1}
\begin{split}
a^{d_1^n} 
& =
a_{(0,0)}^{d_1^n} = 
\Big(\big(a^{b_0^0 b_1^0 }\big)^{d_1}\Big)^{d_1^{n-1}}
\hskip-2mm
= 
\big(a^{b_0^0 b_1^{0+1} }\big)^{d_1^{n-1}}
\hskip-1mm  \\
& 
= 
\big(a^{b_0^0 b_1^{1+1} }\big)^{d_1^{n-2}}
\hskip-1mm= \cdots= 
a^{b_0^0 b_1^n}=
a_{(0,n)}\, ,
\end{split}
\end{equation}
that is, $a_{(0,n)}\in N_1$ for any $n\in \Z$.
\;
On the other hand, using \eqref{EQ elements from <x,y>} 
for $x=a$
and
$y= d_1$,
we can rewrite each $w\in N_1= \langle
a, d_1
\rangle$
as $w=u\cdot v$, where $u$ is a product of some conjugates $a^{\pm d_1^{n_i}}$\!,\, and some $v=d_1^k$\!.
By above construction all those conjugates are in $F$.
Thus, if also $w \in F$, then  $v\in F$.
But that is possible for  $v=1$ only
because $d_1^n$ is a word in stable letters $d,e$ only, and in $\mathscr{A}$ such a word is in $\mathscr{C}$ (and in $F$) only if it is trivial.

\subsection{$\tau \mathcal S$ is benign in $F$}
\label{SU KB is benign} 
Denote $N_2=\langle
a^b, d_0 d_1
\rangle$,\,
$B_2=\langle
a_{(n+1,n)} \mathrel{|} n\in \Z
\rangle
=A_{\tau \mathcal S}$, 
and prove that $F \cap N_2 = B_2$  in $\mathscr{A}$.
We just a little modify the above calculations to get:
\begin{equation}
\label{EQ action of d0d1}
\begin{split}
(a^b)^{(d_0 d_1)^n}
& =
a_{(1,0)}^{(d_0 d_1)^n}
=\big(a_{(1,0)}^{d_0 }\big)^{d_1  \, (d_0 d_1)^{n-1}}
\!\!\!\!
\\
& =  
a_{(1+1,\,0)}^{d_1  \, (d_0 d_1)^{n-1}}
\!\!\!
=a_{(2,\,0+1)}^{(d_0 d_1)^{n-1}}
\!\!\!=a_{(3,2)}^{(d_0 d_1)^{n-2}}
\!\!\!= \cdots 
=a_{(n+1,n)},
\end{split}
\end{equation}
that is, $a_{(n+1,n)}\! \in N_2$ for any $n\in \Z$. 
On the other hand, again using \eqref{EQ elements from <x,y>} 
for $x=a^b$
and
$y= d_0 d_1$, plus the fact that
$(d_0 d_1)^n$ is a word in stable letters $d,e$ only, we using \eqref{EQ elements from <x,y>} get the equality needed.

\subsection{${\mathcal B}=\upsilon(\zeta_{\!1} \Zz , \tau\S)$ is benign in $F$}
\label{SU union zeta Z and tau S is benign} 

Now it is easy to show that  $A_{\mathcal B} = \langle a_f \mathrel{|} f\in \mathcal B \rangle$ is benign in $F$ for ${\mathcal B}=\upsilon(\zeta_{\!1} \Zz , \tau\S)$.
The $\bigast$-construction:
$$
\mathscr{B}
= 
(\mathscr{A} *_{N_1} v_1) *_{\mathscr{A}}
(\mathscr{A} *_{N_2} v_2)
$$
clearly is finitely presented, and in the notations of point~\ref{SU The *-construction} it is noting but\; ${\bigast_{i=1}^2}(\mathscr{A}, N_i, v_i)_{\mathscr{A}}$.  
By Lemma~\ref{LE intersection in bigger group multi-dimensional} the group $\mathscr{B}$, in fact, is equal to $\mathscr{A} *_{N_1, N_2} (v_1, v_2)$. 

Since $A_{\B}=\langle  A_{\zeta_1 \mathcal Z}, A_{\tau \mathcal S}  \rangle$, then 
by Corollary~\ref{CO intersection and join are benign multi-dimensional}\;\eqref{PO 2 CO intersection and join are benign multi-dimensional} the group
$A_{\B}$ is benign in $F$ for the finitely presented group $K_{\B}=\mathscr{B}$ and its finitely generated subgroup 
$L_{\B}=\langle 
F^{\,v_1}, F^{\,v_2}
\rangle$.

\subsection{Writing $K_{\B}$ by generators and defining relations}
\label{SU Writing K_B by generators and defining relations}

From definition of $\xi_m$, $\xi_m$ and $\Xi_m$ given in point~\ref{SU Auxiliary benign free subgroups} we deduce 
$
\Xi_m 
= \big\langle b, c, t_m, t'_m \mathrel{\;|\;} 
b^{t_m}= b^{c^{-m+1}}\!\!,\;
b^{t'_m}= b^{c^{-m}}\!\!,\;
c^{t_m}=c^{t'_m}=c^2
\big\rangle
$.

From definition of $\Theta$, $\mathscr{C}$ and $\mathscr{A}$
in point~\ref{SU Building the group script A} we have $\Theta=\big\langle a, b, c, t_1, t'_1 \mathrel{\;|\;}  
b^{t_1}= b,\;
b^{t'_1}= b^{c^{-1}}\!\!,\;
c^{t_1}=c^{t'_1}=c^2
\big\rangle$, \;
$\mathscr{C}= \big\langle a, b, c, t_1, t'_1, u_1, u_2 \mathrel{\;|\;}  
b^{t_1}= b,\;
b^{t'_1}= b^{c^{-1}}\!\!,\;
c^{t_1}=c^{t'_1}=c^2; 
\text{$u_1$ fixes} $
$\text{ $b^c, t_1, t'_1$};\;\;\;
\text{$u_2$ fixes $a, b, t_1, t'_1$}
\big\rangle$, and: %
\begin{equation}
\label{EQ relations A}
\begin{split}
\mathscr{A} 
& \!=\! 
 \big\langle a, b, c, t_1, t'_1, u_1, u_2, d,e \mathrel{\;|\;}  
b^{t_1}= b,\;
b^{t'_1}= b^{c^{-1}}\!\!,\;
c^{t_1}=c^{t'_1}=c^2; \\[-2pt]
& \hskip11mm 
\text{$u_1$ fixes $b^c\!, t_1, t'_1$};\;
\text{$u_2$ fixes $a, b, t_1, t'_1$}; \\[-2pt]
& \hskip11mm 
\text{$d$ fixes $a^{u_1}\!\!,\; b^{u_1}\!\!,\; c^{u_1}$ and sends $a^{u_2}\!\!,\; b^{u_2}\!\!,\; c^{u_2}$ to $a^{b u_2}\!,\; b^{u_2}\!,\; c^{b u_2}$};\\[-2pt]
& \hskip11mm 
\text{$e$ sends $a,b,c$ to $a,b^c\!,\; c$}
\big\rangle.
\end{split}
\end{equation}

%
%
\noindent
Hence $\mathscr{A}$ is given by $9$
generators and $4+7+6+3=20$ relations.

From definition of $K_{\B}=\mathscr{B}$ given in point~\ref{SU Auxiliary benign free subgroups}:
\begin{equation}
\label{EQ relations B}
\begin{split}
K_{\B}
& =  \big\langle a, b, c, t_1, t'_1, u_1, u_2, d,e, v_1, v_2 \mathrel{\;|\;} 
\text{$20$ relations of $\mathscr{A}$ from \eqref{EQ relations A}};\\[-2pt]
& \hskip11mm 
\text{$v_1$ fixes $a,d^e$};\;\;
\text{$v_2$ fixes $a^b\!,\, d d^e$}
\big\rangle .
\end{split}
\end{equation}
%
%
Hence $K_{\B}$ is given by $11$
generators and $20+4=24$ relations. And the respective finitely presented subgroup in 
$K_{\B}$ is  
$L_{\B}=\langle 
a^{v_1}\!,\, b^{v_1}\!,\, c^{v_1}\!,\,
a^{v_2}\!,\, b^{v_2}\!,\, c^{v_2}
\rangle.$

\section{Construction of $K_{\omega_2 \mathcal B}$ 
}
\label{SE Construction of Ksigma2 omega}

\subsection{The $\omega_2$ operation}
\label{SU The omega2 operation} 

For any subset $\B \subseteq \E$  
Higman defines
$\omega_2(\mathcal B)$ to be the set of all sequences $f\!\in\E$ for which for every $i\!\in\! \Z$ the pair $\big(f(2i\!+\!0),\, f(2i\!+\!1)\big)$ belongs to $\B$ \, \cite{Higman Subgroups of fP groups}. 
In other words, this operation just constructs new sequences $f$ by concatenation of some sequences of length $2$ picked from $\mathcal B$.
Say, if $\mathcal B$ contains
$(0, 0)$, 
$(3, 2)$,
$(9, 8)$, 
then in $\omega_2(\mathcal B)$ we have sequences like 
$(3,2,9,8)$, $(0, 0, 9,8, 3,2, 0, 0, 9,8)$, etc.      

From $A_\B=\langle a_f \mathrel{|} f\in \B\rangle$ of $F$ one may navigate to its overgroup $A_{\omega_2 \B}=\langle a_{\omega_2 \B} \mathrel{|} f\in \B\rangle$ with some important properties of $A_\B$ inherited by $A_{\omega_2 \B}$. 
In this section we prove that if  $A_\B$  is benign in $F$ for an \textit{explicitly} given finitely presented overgroup $K = K_\B$ holding $F$, and for its finitely generated subgroup $L_\B\le K_\B$, then 
$A_{\omega_m \B}$ also is benign in $F$, and the respective $K_{\omega_2 \B}$ and $L_{\omega_2 \B}$ can also be constructed \textit{explicitly}.

Actually, at the end of the current section we are going to take as $\B$ the set 
${\mathcal B}=\upsilon(\zeta_{\!1} \Zz , \tau\S)$ from previous section, with two groups 
$K_\B=\mathscr{B}$ and 
$L_\B=\langle 
F^{\,v_1}, F^{\,v_2}
\rangle$ 
constructed in point~\ref{SU union zeta Z and tau S is benign}, but for the time being it is more comfortable to work with a generic $\B$.

\subsection{The groups $\Xi_2$,  $\Xi_0$, $\Gamma$ and $\mathscr{G}$}
\label{SU The groups Xi Gamma}     

In the notation of point~\ref{SU Auxiliary benign free subgroups} for two values $m\!=\!2$ and $m\!=\!0$ take pairs of isomorphisms $\xi_2, \xi'_2$ and  $\xi_0, \xi'_0$, and construct two HNN-extensions\;
$\Xi_2 \!=\! \langle b,c \rangle *_{\xi_2, \xi'_2} (t_2, t'_2)$ \;and\; $\Xi_0\! = \!\langle b,c \rangle *_{\xi_0, \xi'_0} (t_0, t'_0)$.
As particular cases of Lemma~\ref{LE Ksi}:
\begin{Corollary}
\label{CO Ksi corollary}
In the above notations:
\begin{enumerate}
\item 
\label{PO 1 CO Ksi corollary}
$\langle b,c \rangle \cap \langle b_2, t_2, t'_2\rangle = \langle b_2, b_{3},\ldots\rangle$ holds in $\Xi_2$,

\item 
\label{PO 2 CO Ksi corollary}
$\langle b,c \rangle \cap \langle b_{-1}, t_0, t'_0\rangle = \langle b_{-1}, b_{-2},\ldots\rangle$ holds in $\Xi_0$.
\end{enumerate}
\end{Corollary}

In $\langle b,c \rangle$ form the join 
$B_2 = \langle \ldots b_{-2}, b_{-1};\;
b_2, b_{3},\ldots\rangle$ generated by the above $\langle b_2, b_{3},\ldots\rangle$ and $\langle b_{-1}, b_{-2},\ldots\rangle$. The $\textstyle{\bigast}$-construction:
$$
\mathscr{Z}
= \left(\Xi_2 *_{\langle b_2, t_2, t'_2\rangle} r_1 \right)
\, *_{\langle b,c\rangle}
\left(\Xi_0 *_{\langle b_{-1}, t_0, t'_0\rangle} r_2 \right)
$$
is finitely presented.  We can use Corollary~\ref{CO intersection and join are benign multi-dimensional}\;\eqref{PO 1 CO intersection and join are benign multi-dimensional}, Lemma~\ref{LE join in HNN extension multi-dimensional} and Lemma~\ref{LE intersection in bigger group multi-dimensional} to obtain:

\begin{Corollary}
\label{CO B2 is bening}
The subgroup $B_2$ is benign in $\langle b,c \rangle$. As a finitely presented overgroup of $\langle b,c \rangle$ one can take $\mathscr{Z}$, and as a finitely generated subgroup of $\mathscr{Z}$ one can take $P_2\!=\big\langle\langle b,c\rangle^{r_1} \!,\, \langle b,c\rangle^{r_2} \big\rangle$.
\end{Corollary}

Within this section the letters $g,h,k$ were not yet used, and hence we  may now involve some of them in our free constructions. 
Denote:
\begin{equation}
\label{EQ second Gamma}
\Gamma = \langle b,c \rangle *_{B_2} (g,h,k)
\quad {\rm and} \quad
\mathscr{G} = \mathscr{Z} *_{P_2} (g,h,k)
\end{equation}
with $g,h,k$ fixing the subgroups $B_2$ and $P_2$ respectively. 
The second group $\mathscr{G}$ clearly is finitely presented, and taking into account $\langle b,c\rangle
\cap\,
P_2
 =B_2$ from 
Corollary~\ref{CO B2 is bening} we by 
\cite[Corollary~3.5\;(1)]{Auxiliary free constructions for explicit embeddings}
and \cite[Remark~3.6\;(1)]{Auxiliary free constructions for explicit embeddings}
 have that $\Gamma$ is a subgroup of $\mathscr{G}$, i.e.:

\begin{Lemma} 
\label{LE about alternative Gamma+}
In above notations we have $\langle b,c, g,h,k\rangle=\Gamma$ 
in $\mathscr{G}$.
\end{Lemma}

\subsection{Construction of $\Delta$}
\label{SU Construction of Delta}

The free group $\langle b,c \rangle$ contains a free subgroup $\langle b_i \mathrel{|} i\in \Z\rangle$  of infinite rank, which decomposes into a free product $B_2 * \,\bar B_{2}$ with
$B_2$ mentioned above and with 
and 
$\bar B_{2}=\langle b_0, b_{1}\rangle$.  
In $\Gamma$ pick the subgroup $R=\langle g_f b_f^{-1} \mathrel{|} f\in \mathcal E_2\rangle$.
The letter $a$ was \textit{not} involved in $\Gamma$ above, and we can consider it as a  new stable letter to build the HNN-extension $\Gamma *_R a$ (soon we will see that in $\Gamma *_R a$ three elements $a,b,c$ are free generators for $\langle a,b,c\rangle$, and so we have no conflict with the earlier usage of $F=\langle a,b,c\rangle$ as a free group).

The intersection $\langle b,c \rangle \cap R$ is trivial because the non-trivial words of type $g_f b_f^{-1}$ generate $R$ freely, and so any non-trivial word they generate must involve at least one stable letter $g$, and hence it needs to be outside $\langle b,c \rangle$.
Therefore by \cite[Corollary 3.5\;(1)]{Auxiliary free constructions for explicit embeddings} the subgroup generated in $\Gamma *_R a$ by $\langle b,c \rangle$ together with  $a$ is equal to $\langle b,c \rangle *_{\langle b,c \rangle \,\cap \,R} a = \langle b,c \rangle *_{\{1\}} a
= \langle b,c \rangle *  a=\langle a,b,c \rangle=F$.
Since $a,b,c$ are free generators, the map sending $a,b,c$ to $a,b^{c^2}\!\!\!\!,\;c$ can be continued to an isomorphism  $\rho$ on $F$.  Identifying $\rho$ to a  stable letter $r$ we arrive to:
\begin{equation}
\label{EQ nested form}
\Delta = \big(\Gamma *_R a \big) *_\rho r
= \Big(\big(\langle b,c \rangle *_{B_2} (g,h,k)\big) *_R a \Big) *_\rho r.
\end{equation}

\subsection{Obtaining $F\cap F_{\B}= A_{\omega_2 \B}$ in $\Delta$}
\label{SU Obtaining G cup} 

For $F_{\B}=\langle g_f\!,\, a,\, r \mathrel{|} f\in \B\rangle$ in $\Delta$ it is easy to see that:
\begin{equation}
\label{EQ intersection}
F\cap F_{\B}= A_{\omega_2 \B}
\end{equation}
(check the part ``\textit{$\BB$ is closed under  $\omega_m$}\!''\, in the proof of Theorem 4.4 of \cite{A modified proof for Higman}).
The equality \eqref{EQ intersection} does not \textit{yet} mean that $A_{\omega_2 \B}$ is benign because $\Delta$
in \eqref{EQ nested form} is \textit{not} necessarily finitely presented. Our nearest objective is to replace $\Delta$ by a  finitely presented alternative $\mathscr{D}$, inside which the equality \eqref{EQ intersection} still holds.

\subsection{Presenting $R$ as a join of benign subgroups}
\label{SU Presenting as a join} 

Let us brake down the subgroup $R$ 
of $\Gamma$ (and hence of $\mathscr{G}$, see Lemma~\ref{LE about alternative Gamma+}) 
to a join of three subgroups, each benign in  $\mathscr{G}$.  
Denote $\Phi=\langle b_0,b_{1}, g, h_0,h_{1} \rangle$, and notice that:

\begin{Lemma}
\label{LE new free subgroup}
$\Phi$ is freely generated by  $5$ elements $b_0,b_{1}, g, h_0,\,h_{1}$ in $\Gamma$, and hence in $\mathscr{G}$.
\end{Lemma} 

\begin{proof}  
Since $\bar B_{2}=\langle b_0, b_{1}\rangle$ has trivial intersection with $B_2$, we in $\Gamma$, and hence in $\mathscr{G}$ by \cite[Corollary 3.5\;(1)]{Auxiliary free constructions for explicit embeddings} and by \cite[Corollary 3.6]{Auxiliary free constructions for explicit embeddings} have:
$
\langle b_0,b_{1},\; g,h,k\rangle
=\bar B_{2} *_{\bar B_{2}   \;\cap\;   B_2} (g,h,k) = \bar B_{2} *_{{\1}} (g,h,k)
=\bar B_{2} * \langle g,h,k \rangle
$,
and the latter simply is a free group.
Since $g, h_0,h_{1}$ generate a free subgroup inside $\langle g,h,k \rangle$, they together with $b_0,b_{1}$ generate a free subgroup (of rank $2\!+\!1\!+\!2\!=\!5$)
inside $\langle b_0,b_{1},\; g,h,k\rangle$.
\end{proof}

\medskip

For any $f\!=\!(j_0, j_{1})\in \mathcal E_2$ following the notation in \ref{SU Integer functions f} put $f^+=(j_0,\, j_{1}\!+\!1)$. In this notation $\Gamma$ contains the element 
$g_{f^+}^{\vphantom{1}} \!
\cdot 
b_{1}^{-1} 
\cdot 
g_f^{-1}$, such as, 
$g^{h_0^{2} h_1^{\boldsymbol 6} }
\cdot
b_{1}^{-1} 
\cdot
g^{-\,h_0^{2} h_1^{5}  }$  
for $f=(2,5)$.
In a similar way, for any $f=(j_0)\in \mathcal E_1$ denote $f^+=(j_0+1)$, i.e., we just add $1$ to the only coordinate $f(0)$ of $f$. Then $\Gamma$ contains the element 
$g_{f^+}^{\vphantom{1}} \!
\cdot 
b_{0}^{-1} 
\cdot 
g_f^{-1}$, such as, 
$g^{h_0^{\boldsymbol 3} }
\cdot
b_{0}^{-1} 
\cdot
g^{-\,h_0^{2} }$  
for the sequence $f=(2)$ of length $1$.
Define two subgroups:
$$
V_{\mathcal E_2}
=\Big\langle 
g_{f^+}^{\vphantom{1}} \!
\cdot 
b_{1}^{-1} 
\cdot 
g_f^{-1} \mathrel{|} f\in \mathcal E_2
\Big\rangle
\quad {\rm and} \quad
V_{\mathcal E_1}
=\Big\langle 
g_{f^+}^{\vphantom{1}} \!
\cdot 
b_{0}^{-1} 
\cdot 
g_f^{-1} \mathrel{|} f\in \mathcal E_1
\Big\rangle
$$ 
in $\Gamma$, and hence in $\mathscr{G}$, to establish a property for them:
 
\begin{Lemma}
\label{LE VEm is bening}
Each of  
$V_{\mathcal E_2}$ and $V_{\mathcal E_1}$ is benign subgroups in $\mathscr{G}$ for the some explicitly given finitely presented group and its finitely generated subgroup.
\end{Lemma}

\begin{proof}  
By Lemma~\ref{LE new free subgroup} 
the elements $b_{1}, g, h_0, h_{1}$ are \textit{free} generators for the $4$-generator subgroup
$\langle b_{1}, g, h_0, h_{1} \rangle$
in $\Phi$. Hence any of the following two maps can be continued to an isomorphism on the free group $\langle b_{1}, g, h_0, h_{1} \rangle$:
$$
\lambda_{1,0} 
\quad{\rm sends}\quad
b_{1}, g, h_0, h_{1}
\quad{\rm to}\quad
b_{1}, g^{h_0}, h_0, h_{1},\;\,
$$
$$
\lambda_{1,1} 
\quad{\rm sends}\quad
b_{1}, g, h_0, h_{1}
\quad{\rm to}\quad
b_{1}, g^{h_1}, h_0^{h_1}, h_{1}.
$$
Similarly, the elements $b_{0}, g, h_0$ are \textit{free} generators for the $3$-generator subgroup
$\langle b_{0}, g, h_0 \rangle$
of $\Phi_2$, and the following can be continued to an isomorphism on $\langle b_{0}, g, h_0 \rangle$:
$$
\lambda_{0,0} 
\quad{\rm sends}\quad
b_{0}, g, h_0
\quad{\rm to}\quad
b_{0}, g^{h_0}, h_0\,.
$$
For these isomorphisms pick three stable letters 
$l_{1,0},\, l_{1,1}, \,l_{0,0}$ to respectively construct the following two finitely presented HNN-extensions:
$$
\Lambda_2=
\mathscr{G}
*_{\,\lambda_{1, 0}\,,\,\lambda_{1, 1}} (l_{1,0},\; l_{1,1})
\quad {\rm and} \quad
\Lambda_1=
\mathscr{G}
*_{\,\lambda_{0,\; 0}} l_{0,0}\;.
$$
Conjugation of elements  $g_{f^+}^{\vphantom{1}} \!
\cdot 
b_{1}^{-1}  
\cdot 
g_f^{-1}\in \Lambda_2$ by  $l_{1,0}$ and $l_{1,1}$ is easy to understand: 
$l_{1,0}$ just adds $1$  to the coordinate $f(0)$ of $f$,
and $l_{1, 1}$ just adds $1$ to the coordinate $f(1)$ of $f$.
For instance, for $f=(2,5)$ and for $l_{1,1}$ we from the above definition of $\lambda_{1,1}$ easily get: 
\begin{equation}
\label{EQ example with l CASE 2}
\begin{aligned}
\left(
g_{f^+}^{\vphantom{1}} \!
\cdot 
b_{1}^{-1} 
\cdot 
g_f^{-1}
\right)^{l_{1,1}}
\!\!
&=\left(g^{h_{1}}\right)^{
\left(h_0^{h_{1}}\right)^{2}
h_1^{6} }
\cdot
b_{1}^{-1} 
\! \cdot
\left(g^{h_{1}}\right)^{-\,
\left(h_0^{h_{1}}\right)^{2}
h_1^{5}
} \! =  g^{h_1 \cdot\, h_1^{-1}  h_0^2 h_1 
h_1^{6} }  
\cdot
b_{1}^{-1} 
\! \cdot
g^{-h_1 \cdot\, h_1^{-1}  h_0^2 h_1 
h_1^{5} }\\[-1mm]
& = g^{h_0^{2} h_1^{\boldsymbol{7}} }
\cdot
b_{1}^{-1} 
\cdot
g^{-\,h_0^{2} h_1^{\boldsymbol{6}}  }
=g_{f'^{\;+}}^{\vphantom{1}} \!
\cdot 
b_{1}^{-1} 
\! \cdot 
g_{f'}^{-1}
\in V_{\mathcal E_2}
\end{aligned}
\end{equation}
with $f'=(2,\,5\!+\!1)=(2,\boldsymbol{6})\in \mathcal E_2$.
Similarly verify conjugation of $g_{f^+}^{\vphantom{1}} \!
\cdot 
b_{1}^{-1}  
\cdot 
g_f^{-1}\in \Lambda_2$ by $l_{1,0}$.

And it is even simpler to see that conjugation of  $g_{f^+}^{\vphantom{1}} \!
\cdot 
b_{0}^{-1}  
\cdot 
g_f^{-1}\in \Lambda_1$ by the stable letter $l_{0,0}$ just adds $1$  to the only coordinate $f(0)$ of $f$.
Say, when $f=(2)$ we trivially have:
\begin{equation}
\label{EQ example with l CASE 1}
\begin{aligned}
\left(
g_{f^+}^{\vphantom{1}} \!
\cdot 
b_{0}^{-1} 
\cdot 
g_f^{-1}
\right)^{l_{0,0}}
&=\left(g^{h_{0}}\right)^{
h_0^{3} }
\cdot
b_{0}^{-1} 
\! \cdot
\left(g^{h_{0}}\right)^{-\,
h_0^{2} 
} 
=  g^{h_0 \cdot\,  h_0^3  }  
\cdot
b_{0}^{-1} 
\! \cdot
g^{-h_0 \cdot\,  h_0^2  }\\[-1mm]
& = g^{h_0^{\boldsymbol{4}} }
\cdot
b_{0}^{-1} 
\cdot
g^{-\,h_0^{\boldsymbol{3}}  }
=g_{f'^{\;+}}^{\vphantom{1}} \!
\cdot 
b_{0}^{-1} 
\! \cdot 
g_{f'}^{-1}
\in V_{\mathcal E_1}
\end{aligned}
\end{equation}
with $f'=(2\!+\!1)=(\boldsymbol{3})\in \mathcal E_1$.

We see that conjugations by the letters $l_{1,0}$ and $l_{1,1}$ keep the elements from $V_{\mathcal E_2}$ inside $V_{\mathcal E_2}$, and 
conjugation by the letter $l_{0,0}$  keeps the elements from $V_{\mathcal E_1}$ inside $V_{\mathcal E_1}$.

\medskip

We in point~\ref{SU Integer functions f} agreed that for our purposes we may append zero members to any sequence $f$, and this does not change the respective elements $b_f, a_f, h_f, g_f$. Hence, interpret the zero sequence as $f_{0}=(0,0)\in \E_2$, and rewrite the product $g^{h_{1}} 
\!
\cdot 
b_{1}^{-1} 
\cdot 
g^{-1}$ as
$g_{f_0^+}^{\vphantom{1}} \!
\cdot 
b_{1}^{-1} 
\cdot 
g_{f_0}^{-1}
$ for this sequence $f_{0}=(0,0)$.
Applying the conjugate collection process of point~\ref{SU The conjugates collecting process} for $\X=\{
g^{h_{1}} 
\!
\cdot 
b_{1}^{-1} 
\cdot 
g^{-1}\}
$ and for 
$\Y=\{l_{1, 0},l_{1,1}\}$
we see that any element $w$ from $\langle \X
, \Y \rangle$ inside $\Lambda_2$ is a product of elements of $g_{f^+}^{\vphantom{1}} \!
\cdot 
b_{1}^{-1} 
\cdot  
g_f^{-1}$ (for certain sequences $f\in \mathcal E_2$) and of certain powers of the stable letters $l_{1, 0},l_{1, 1}$. And $w$ is inside $\mathscr{G}$ if and only if all those powers are cancelled out in the normal form, and $w$ in fact is in $V_{\mathcal E_2}$, that is, denoting $L_2 = \langle
g^{h_{1}} 
\!
\cdot 
b_{1}^{-1} 
\cdot 
g^{-1}\!\!,\;\; l_{1,0},\;l_{1,1}
\rangle$ we have $\mathscr{G}\, \cap \, L_2
\subseteq V_{\mathcal E_2}$.

On the other hand, for \textit{any} $f \in \E_2$ it is very easy to obtain $g_{f^+}^{\vphantom{1}} \!
\cdot 
b_{1}^{-1}  
\cdot 
g_f^{-1}$ via conjugations of  $g^{h_{1}} 
\!
\cdot 
b_{1}^{-1} 
\cdot 
g^{-1}$ by $l_{1,0}$ and $l_{1,1}$, see \eqref{EQ example with l CASE 2} and \eqref{EQ example with l CASE 1}. For instance, for $f=(2,5)$ we have:
$$
g_{f^+}^{\vphantom{1}} \!
\cdot 
b_{1}^{-1}  
\cdot 
g_f^{-1} = 
\left(
g^{h_{1}} 
\!
\cdot 
b_{1}^{-1} 
\cdot 
g^{-1}
\right)^{l_{1,0}^2 \, \cdot \;l_{1,1}^5}\!.
$$
Therefore, $
\mathscr{G}\, \cap \, L_2
= V_{\mathcal E_2},
$
i.e., $V_{\mathcal E_2}$ is benign in $\mathscr{G}$ for the above finitely presented group $K=\Lambda_2$ and for its $3$-generator subgroup $L=L_2$.

\medskip

Next interpret the zero sequence as $f_{0}=(0)\in \E_1$, and write 
$g_{f_0^+}^{\vphantom{1}} \!
\cdot 
b_{0}^{-1} 
\cdot 
g_{f_0}^{-1}
= 
g^{h_{0}} 
\!
\cdot 
b_{0}^{-1} 
\cdot 
g^{-1}
$\!\!.\;
By the conjugate collection process of point~\ref{SU The conjugates collecting process} for $\X=\{
g^{h_{0}} 
\!
\cdot 
b_{0}^{-1} 
\cdot 
g^{-1}\}
$ and for 
$\Y=\{l_{0, 0}\}$
we see that any element $w \in \langle \X
, \Y \rangle \le \Lambda_1$ is a product of elements of $g_{f^+}^{\vphantom{1}} \!
\cdot 
b_{0}^{-1} 
\cdot 
g_f^{-1}$ (for certain $f\in \mathcal E_1$) and of certain powers of the stable letter $l_{0, 0}$. So $w$ is inside $\mathscr{G}$ if and only if all those powers are cancelled out, and $w$ actually is in $V_{\mathcal E_1}$, that is, denoting $L_1 = \langle
g^{h_{0}} 
\!
\cdot 
b_{0}^{-1} 
\cdot 
g^{-1}\!\!,\;\; l_{0, 0}
\rangle$ we have:
$
\mathscr{G}\, \cap \, L_1
\subseteq V_{\mathcal E_1}.
$
On the other hand, for \textit{any} $f \in \E_1$ it is very easy to obtain $g_{f^+}^{\vphantom{1}} \!
\cdot 
b_{0}^{-1}  
\cdot 
g_f^{-1}$ via conjugations of  $g^{h_{0}} 
\!
\cdot 
b_{0}^{-1} 
\cdot 
g^{-1}$ by $l_{0,0}$. 
Hence, $\mathscr{G}\, \cap \, L_1 = V_{\mathcal E_1}$, i.e., $V_{\mathcal E_1}$ is benign in $\mathscr{G}$ for the above finitely presented group $K=\Lambda_1$ and for its $2$-generator subgroup $L=L_1$.
\end{proof}

In addition to the groups set in the above proof define the auxiliary groups  
$V_{\mathcal E_0}=L_0=\langle g \rangle$ and $\Lambda_0=\mathscr{G}$. Since this $V_{\mathcal E_0}$ already is finitely generated, it trivially is benign in finitely generated $\mathscr{G}$ for the stated finitely presented $\Lambda_0$ and for the finitely generated $L_0$.

\medskip
The collected information outputs:

\begin{Lemma}
\label{LE represent L}
$R=\langle g_f b_f^{-1} \mathrel{|} f\in \mathcal E_2\rangle$  is a benign subgroup in $\mathscr{G}$ for  some explicitly given finitely presented group and its finitely generated subgroup.
\end{Lemma}

\begin{proof}
First show that
$R$
is generated by its three subgroups $V_{\mathcal E_0}, V_{\mathcal E_1}, V_{\mathcal E_2}$.
Denote  
$Z_{\mathcal E_2}=\langle g_f b_f^{-1} \mathrel{|} f\in \mathcal E_2\rangle=R$,\; $Z_{\mathcal E_1}=\langle g_f b_f^{-1} \mathrel{|} f\in \mathcal E_1\rangle$ and $Z_{\mathcal E_0}=\langle g \rangle$.
It is very easy to see that 
$\langle Z_{\E_{1}}, V_{\E_2} \rangle = Z_{\E_2}$ 
and 
$\langle Z_{\E_{0}}, V_{\E_1} \rangle = Z_{\E_1}$,
see details in \cite{A modified proof for Higman} based on an original idea from \cite{Higman Subgroups of fP groups}.
And, thus, 
$
R=
Z_{\E_{2}} 
= \langle Z_{\E_{1}}, V_{\E_{2}} \rangle 
= \langle Z_{\E_{0}}, 
V_{\E_{1}}, V_{\E_{2}} \rangle
= \langle V_{\E_{0}}, 
V_{\E_{1}}, V_{\E_{2}} \rangle.
$ 

Next, by the above each of 
$V_{\mathcal E_i}$, $i=0,1,2$, is  benign in $\mathscr{G}$ for an explicitly given finitely presented group $\Lambda_i$ and its finitely generated subgroup $L_i$.

It remains to load these components into \eqref{EQ nested Theta  multi-dimensional short form}, and to apply Corollary~\ref{CO intersection and join are benign multi-dimensional}\;\eqref{PO 2 CO intersection and join are benign multi-dimensional} to get the following finitely presented overgroup holding $\mathscr{G}$:
\begin{equation}
\label{EQ adapted construction with Theta} 
\mathscr{F}=
\Big((\Lambda_0 *_{L_0} s_0) *_{\mathscr{G}} (\Lambda_1 *_{L_1} s_1) \Big) *_{\mathscr{G}}  (\Lambda_2 *_{L_2} s_2)
=
{\textstyle \bigast_{i=0}^{2}}(\Lambda_i, L_i, s_i)_{\mathscr{G}}.
\end{equation}
and its finitely generated subgroup 
$\mathscr{H}=\big\langle \mathscr{G}^{\,s_0}
, \mathscr{G}^{\,s_1}
, \mathscr{G}^{\,s_2}
\big\rangle$ for which $\mathscr{G} \cap \mathscr{H} = R$ holds.
\end{proof}

\subsection{Construction of finitely presented group for $\omega_2 \B$
}
\label{SU Construction of finitely presented group for omega_2 B}  

Observe that in construction of $\mathscr{F}$ we never used the letter $a\in F$. Hence, in analogy with construction of $\Gamma *_R a$ in  subsection~\ref{SU Construction of Delta}, we can build the HNN-extension $\mathscr{F} *_\mathscr{H} a$ using $a$ as a stable letter fixing $\mathscr{H}$. 
Since $\mathscr{F}$ of \eqref{EQ adapted construction with Theta} is finitely presented, and $\mathscr{H}$ is finitely generated, $\mathscr{F} *_\mathscr{H} a$ is finitely presented. 

Inside $\mathscr{F} *_\mathscr{H} a$ the elements $a,b,c$ generate the same \textit{free} subgroup discussed in \ref{SU Construction of Delta}, and so we can again define an isomorphism $\rho$ sending $a,b,c$ to $a,b^{c^2}\!\!\!\!,\;c$ together with the 
 \textit{finitely presented} analog of $\Delta$ from 
\eqref{EQ nested form}:
\begin{equation}
\label{EQ nested form analog}
\mathscr{D}
= \big(\mathscr{F} *_\mathscr{H} a \big) *_\rho r.
\end{equation}
For any $\B \subseteq \E_2$ we in analogy with subsection~\ref{SU Construction of Delta} introduce $F_{\B}=\langle g_f\!,\, a,\, r \mathrel{|} f\in \B\rangle$ in $\mathscr{D}$. But since $F_{\B}$ is in the subgroup $\Delta$ of $\mathscr{D}$ already, we for $\mathscr{D}$ have the analog of \eqref{EQ intersection}:
\begin{equation}
\label{EQ intersection analog}
F\cap F_{\B}= A_{\omega_2 \B} \text{\; in \;} \mathscr{D}.
\end{equation}
         
As we saw in Section~\ref{SE Construction of KB}, for the particular set ${\mathcal B}=\upsilon(\zeta_{\!1} \Zz , \tau\S)$ the group $A_{\mathcal B} = \langle a_f \mathrel{|} f\in \mathcal B \rangle$ is benign in $F=\langle a, b, c \rangle$.
As a finitely presented overgroup for $F$ we in \ref{SU union zeta Z and tau S is benign} constructed the group $K_\B = \mathscr{B}$ with its finitely generated subgroup
 $L_\B = \langle 
F^{\,v_1}, F^{\,v_2}
\rangle$, with 
$F \cap L_\B = A_{\mathcal B}$.

\medskip
$\mathscr{D}$ was built  via free constructions 
by adjoining to  $\langle 
b,c
\rangle$ some new letters
$t_2, t'_2, t_0, t'_0, r_1, r_2,$ 
$g,h,k,\, l_{0,0},\, l_{1,0},\, l_{1,1},
s_0, s_1, s_2,\, a, r,
$
while in Section~\ref{SE Construction of KB} we built $K_\B = \mathscr{B}$ by adjoining to \textit{the same} $\langle 
b,c
\rangle$ the new letters
$
t_1, t'_1, a, u_1, u_2, d, e, v_1, v_2. 
$
We also noticed that the letters $a,b,c$ freely generate $F$ in both  $\mathscr{B}$ and $\mathscr{D}$.
Hence, we can construct the finitely presented amalgamated product $\mathscr{B} *_F \mathscr{D}$
generated by the $2+17 + 9-1 = 27$  letters: $b,c$ plus those in two lines of letters above (we subtract $1$ as the letter $a$ accrued twice).

\medskip
The subgroup 
$A_{\B}$ is benign also in $\mathscr{D}$. Indeed, the group $\mathscr{B} *_F \mathscr{D}$ is finitely presented, and since $L_\B \cap F = A_\B$
and
$A_\B \cap F = A_\B$, we can apply to the subgroup $L_\B = \langle\, L_\B ,\, A_\B\rangle$ of $\mathscr{B} *_F \mathscr{D}$ 
\cite[Corollary~3.2\;(3)]{Auxiliary free constructions for explicit embeddings}  
to get that  $\mathscr{D} \cap L_\B = A_{\B}$.
Next, being finitely generated $\langle b,c \rangle$ is benign in $\mathscr{D}$ for the finitely presented $\mathscr{D}$ and for the finitely generated $\langle b,c \rangle$, see Remark~\ref{RE finite generated is benign}.
Hence by Corollary~\ref{CO intersection and join are benign multi-dimensional}\;\eqref{PO 2 CO intersection and join are benign multi-dimensional}  
the join $\big\langle A_{\B}, \langle b,c \rangle \big\rangle = F_{\B}$ is benign in $\mathscr{D}$.
As its finitely presented overgroup we may take:
\begin{equation}
\label{EQ definition of L}
\mathscr{L}=
\big((\mathscr{D} *_F \mathscr{B})*_{L_\B} q_1 \big) 
\;*_{\mathscr{D}}\,
\big(\mathscr{D} *_{\langle b,c \rangle} q_2\big),
\end{equation}
and as a finitely generated subgroup we may take $L' = \big\langle \mathscr{D}^{p_1}\!,\; \mathscr{D}^{p_2} \big\rangle$.

$F$ is benign in $\mathscr{D}$ for the finitely presented $\mathscr{D}$ and for the finitely generated $F$.
Hence by
Corollary~\ref{CO intersection and join are benign multi-dimensional}\;\eqref{PO 1 CO intersection and join are benign multi-dimensional} 
the intersection 
$F\cap F_{\B}= A_{\omega_2 \B}$ is benign in $\mathscr{D}$ for the finitely presented:
\begin{equation}
\label{EQ explicite KomegaB}
K_{\omega_2 \B}=(\mathscr{L} *_{L'} p_3) \;*_{\mathscr{D}} \,
(\mathscr{D} *_F p_4),
\end{equation}
generated by $31$ letters:
\begin{equation}
\label{EQ generators of K omega2}
\begin{split}
& b, c, t_2, t'_2, t_0, t'_0, r_1, r_2,
g,h,k,\, l_{0,0},\, l_{1,0},\, l_{1,1},
s_0, s_1, s_2,\, a, r, \\
& \hskip15.5mm t_1, t'_1, u_1, u_2, d,e, v_1, v_2,\;\; q_1, q_2, q_3, q_4
\end{split}
\end{equation}
and for the finitely generated subgroup: 
\begin{equation}
\label{EQ explicite LomegaB}
L_{\omega_2 \B} = \mathscr{D}^{\,p_3 p_4}
\end{equation}
in the above $K_{\omega_2 \B}$,
that is, $\mathscr{D} \cap L_{\omega_2 \B} =A_{\omega_2 \B}$ in $K_{\omega_2 \B}$.
But since $F \le \mathscr{D}$ and $A_{\omega_2 \B} \le F$, we conclude that $F \cap L_{\omega_2 \B} = A_{\omega_2 \B}$ also holds in $K_{\omega_2 \B}$.

\subsection{Writing $K_{\omega_2 \B}$ by generators and defining relations}
\label{SU Writing K_omega B by generators and defining relations}
Using presentation of $\Xi_m$ in point~\ref{SU Writing K_B by generators and defining relations} (for $m=2$ and $m=0$), and definitions of $\mathscr{Z}$ and $\mathscr{G}$ in point~\ref{SU The groups Xi Gamma} we have:
\begin{equation}
\label{EQ relations Z}
\begin{split}
\mathscr{Z}
& = \big\langle
b, c, t_2, t'_2, t_0, t'_0, r_1, r_2 
\mathrel{\;\;|\;\;}
b^{t_2}= b^{c^{-1}}\!\!,\;
b^{t'_2}= b^{c^{-2}}\!\!,\;
b^{t_0}= b^{c}\!,\;
b^{t'_0}= b,\;\\
& \hskip6mm
c^{t_2}=c^{t'_2}=c^{t_0}=c^{t'_0}=c^2;\;\;\;
\text{$r_1$ fixes $b_2, t_2, t'_2$};\;\;\;
\text{$r_2$ fixes $b_{-1}, t_0, t'_0$}
\big\rangle.
\end{split}
\end{equation}

\begin{equation}
\label{EQ relations G}
\begin{split}
\mathscr{G}
& = \big\langle
b, c, t_2, t'_2, t_0, t'_0, r_1, r_2,
g,h,k
\mathrel{\;\;|\;\;}
\text{$14$ relations of $\mathscr{Z}$ from \eqref{EQ relations Z}};\\
& \hskip6mm
\text{each of $g,h,k$ fixes 
$b^{r_1}\!,\; c^{r_1}\!,\; b^{r_2}\!,\; c^{r_2}$}
\big\rangle.
\end{split}
\end{equation}

Using 
$L_0$, $\Lambda_0$,
$L_1$, $\Lambda_1$,
$L_2$, $\Lambda_2$ and 
$\mathscr{F}$
in point~\ref{SU Presenting as a join} (in particular, using  $\mathscr{F}$ in \eqref{EQ adapted construction with Theta}), write:
%
\begin{equation}
\label{EQ relations F}
\begin{split}
\mathscr{F}
& = \big\langle
b, c, t_2, t'_2, t_0, t'_0, r_1, r_2,
g,h,k,\, l_{0,0},\, l_{1,0},\, l_{1,1},
s_0, s_1, s_2
\mathrel{\;\;|\;\;}\\
& \hskip11mm
\text{$14$ relations of $\mathscr{Z}$ from \eqref{EQ relations Z}};\;\;\;
\text{each of $g,h,k$ fixes 
$b^{r_1}\!,\; c^{r_1}\!,\; b^{r_2}\!,\; c^{r_2}$};\\
& \hskip11mm
\text{$l_{0,0}$ sends $
b, g, h$ to $
b, g^{h}\!, h$};\;\;\;
\text{$l_{1,0}$ sends $
b^c\!, g, h, h^k$ to $b^c\!, g^{h}\!, h, h^k$};\\
& \hskip11mm
\text{$l_{1,1}$ sends $
b^c\!,\; g, h, h^k$ to $
b^c\!,\; g^{h^c}\!\!\!,\; h^{h^k}\!\!\!,\; h^k$};\\
& \hskip11mm
\text{$s_0$ fixes $g$};\;\;\;
\text{$s_1$ fixes $g^{h} 
b^{-1} 
g^{-1}\!\!,\; l_{0, 0}$};\;\;\;
\text{$s_2$ fixes $g^{h^k} 
b^{-c} 
g^{-1}\!\!,\;\; l_{1,0},\;\;l_{1,1}$}
\big\rangle.
\end{split}
\end{equation}
$\mathscr{F}$ has $17$ generators and $14+12+11+6=43$ relations.

Using definition of $\mathscr{H}$ in point~\ref{SU Presenting as a join} 
and
definition of $\mathscr{D}$ in point~\ref{SU Construction of finitely presented group for omega_2 B}, we have:
\begin{equation}
\label{EQ relations D}
\begin{split}
\mathscr{D}
& = \big\langle
b, c, t_2, t'_2, t_0, t'_0, r_1, r_2,
g,h,k,\, l_{0,0},\, l_{1,0},\, l_{1,1},
s_0, s_1, s_2,\, a, r
\mathrel{\;\;|\;\;}\\
& \hskip11mm
\text{$43$ relations of $\mathscr{F}$ from \eqref{EQ relations F}};\\
& \hskip11mm
\text{$a$ fixes conjugates of 
$b, c, t_2, t'_2, t_0, t'_0, r_1, r_2,
g,h,k$ by each of $s_0, s_1, s_2$};\\
& \hskip11mm
\text{$r$ sends $a,b,c$ to $a,b^{c^2}\!\!\!\!,\;c$}
\big\rangle.
\end{split}
\end{equation}

\noindent
$\mathscr{D}$ has $17+2=19$ generators and $43 + 11 \cdot 3 + 3 = 79$ relations.

From definition of $\mathscr{L}$ in \eqref{EQ definition of L} and from definition of $L_\B$ in point~\ref{SU Construction of finitely presented group for omega_2 B}:
\begin{equation}
\label{EQ relations L}
\begin{split}
\mathscr{L}
&  = \big\langle
\text{$19$ generators of $\mathscr{D}$ from \eqref{EQ relations D}};\; t_1, t'_1, u_1, u_2, d,e, v_1, v_2,\;\;\, q_1, q_2
\mathrel{\;\;|\;\;}\\
& \hskip11mm
\text{$79$ relations of $\mathscr{D}$ from \eqref{EQ relations D}};\;\;\; 
\text{$24$ relations of $\mathscr{B}$ from \eqref{EQ relations B}};\\
& \hskip11mm
\text{$q_1$ fixes 
$a^{v_1}\!,\, b^{v_1}\!,\, c^{v_1}\!,\, a^{v_2}\!,\, b^{v_2}\!,\, c^{v_2}$;\; 
$q_2$ fixes $b,c$}
\big\rangle.
\end{split}
\end{equation}

$\mathscr{L}$ has $19+10=29$ generators and $79 + 24 + 6+2 = 111$ relations.

Finally, from definition of 
$K_{\omega_2 \B}$ in \eqref{EQ explicite KomegaB} and of $L' = \big\langle \mathscr{D}^{p_1}\!,\; \mathscr{D}^{p_2} \big\rangle$
in point~\ref{SU Construction of finitely presented group for omega_2 B} we have:
\begin{equation}
\label{EQ relations K omega2 B}
\begin{split}
K_{\omega_2 \B}
& = \big\langle
\text{$29$ generators of $\mathscr{L}$ from \eqref{EQ relations L}};\; q_3, q_4
\mathrel{\;\;|\;\;}\\
& \hskip11mm
\text{$111$ relations of $\mathscr{L}$ from \eqref{EQ relations L}};\\
& \hskip11mm
\text{$q_3$ fixes the conjugates of 
$b, c, t_2, t'_2, t_0, t'_0, r_1, r_2,
g,h,k,$}\\
& \hskip11mm 
\text{$l_{0,0},\, l_{1,0},\, l_{1,1},
s_0, s_1, s_2,\, a, r$ by each of $p_1,\, p_2$}; \;\;\;
\text{$q_4$ fixes $a,b,c$}
\big\rangle.
\end{split}
\end{equation}
$K_{\omega_2 \B}$ has $29+2=31$ generators and $111+19\cdot 2 + 3= 152$ relations. As $L_{\omega_2 \B}$ we take: 
$$
L_{\omega_2 \B} = \mathscr{D}^{\,p_3 p_4}=
\big\langle 
b, c, t_2, t'_2, t_0, t'_0, r_1, r_2,
g,h,k,\, l_{0,0},\, l_{1,0},\, l_{1,1},
s_0, s_1, s_2,\, a, r
\big\rangle^{\,p_3 p_4}\!.
$$

\section{Construction of $K_{\sigma^{\,-1} \omega_2 \mathcal B}$ and $K_{\mathcal C}$
}
\label{SE Construction of Ksigma and KC}


\subsection{The ``shifted'' version 
of $\omega_2\mathcal B$}
\label{SU the shifted version} 

Let $\sigma^{-1} \omega_2 \mathcal B$ denote the set of all functions from $\omega_2 \mathcal B$ ``shifted one coordinate to the left'' or, more precisely, $f'\in \sigma^{-1} \omega_2 \mathcal B$ if and only if there is an $f \in \omega_2 \mathcal B$ such that $f'(i)=f(i\!+\!1)$ for all $i\in \Z$.
For example, if $f=(0,5,9,8)\in \omega_2\mathcal B$, then $f'=(5,9,8) \in \sigma^{-1} \omega_2\mathcal B$;
and if $f=(3,2,9,8)\in \omega_2\mathcal B$, then $f'\in \sigma^{-1} \omega_2\mathcal B$ is the function sending the integers $-1, 0, 1, 2$ to $3,2,9,8$ (and all other integers to $0$).

It would not be hard to modify the construction of previous two sections to show that $A_{\sigma^{-1} \omega_2 \mathcal B}$ also is benign in $F$ for ${\mathcal B}=\upsilon(\zeta_{\!1} \Zz , \tau\S)$. However, it will be by far simpler to achieve that aim by 
adding some extra step over the work already done for $K_{\omega_2 \B}$ and $L_{\omega_2 \B}$.

\subsection{The benign image $A_{\sigma^{-1} \omega_2 \mathcal B} $}
\label{SU benign A sigma omega B}

Within this section for shorter notation put
$K=K_{\omega_2 \B}$
and
$L=L_{\omega_2 \B}$ to be the groups found in point~\ref{SU Construction of finitely presented group for omega_2 B}. 
From previous section we know that $F=\langle a,b,c \rangle$ is embeddable into $K$, and $F \cap L = A_{\omega_2\mathcal B}$.

For the set of $31$ letters $b,c;\ldots; p_4$ from \eqref{EQ generators of K omega2} introduce some new $31$ letters $\bar b,\bar c;\ldots; \bar p_4$ letters (bars added on each letter). Then 
construct a copy $\bar K$ in the same manner as $K=K_{\omega_2 \B}$, using the same free constructions on these new letters, i.e., start from $\langle \bar b,\bar c \rangle$, then build its HNN-extension by the stable letters $\bar t_1, \bar t'_1$, where $\bar t_1$
sends $\bar b,\bar c$ to 
$\bar b_0,\bar c^2$, and 
$\bar t'_1$
sends $\bar b,\bar c$ to 
$\bar b_{-1},\bar c^2$, etc...
Clearly, $\bar  K$ contains the copy $\bar F=\langle \bar a,\bar b,\bar c \rangle$ of $F$, and this $\bar F$ has a benign subgroup $\bar A_{\omega_2\mathcal B}=\langle \bar a^{\bar b_f} \mathrel{|} f\in \omega_2\mathcal B \rangle$ such that $\bar F \cap \bar  L=\bar A_{\omega_2\mathcal B}$ for the respective finitely generated $\bar L$.

Define a function $\chi:\bar F \to F$ sending $\bar a,\bar b,\bar c$ \;to\; $a,b_{-1},c$ (where $b_{-1}=b^{c^{-1}}$). 
If we for any word $w\in F$ set  $\bar w\in \bar F$ to be the analog of $w$ just written ``with bars'', then $\chi(\bar w)\in F$ is $w$ in which each $b_i$ is replaced by $b_{i-1}$. 
Turning to examples in \ref{SU the shifted version} above, if $f=(0,5,9,8)\in \omega_2\mathcal B$, then $\bar b_f=\bar b_0^0 \bar b_1^5 \bar b_2^9 \bar b_3^8$,\, $\bar a_f =  \bar b_f^{-1} \bar a\, \bar b_f$ 
and $\chi: \bar a_f \to a_{f'}=
a^{b_0^5 b_1^9 b_2^8 }
$ with $f'=(5,9,8) \in \sigma^{-1} \omega_2\mathcal B$.
And if $f=(3,2,9,8)\in \omega_2\mathcal B$, then  $\chi: \bar a_f \to a_{f'}$ where $f'\in \sigma^{-1} \omega_2\mathcal B$ is the function sending the integers $-1, 0, 1, 2$ to $3,2,9,8$ (and all other integers to $0$) see \ref{SU the shifted version}. Evidently:

\begin{Lemma}
\label{LE sigma shifted}  
In the above notations  $A_{\sigma^{\,-1} \omega_2 \mathcal B}=\chi(\bar A_{\omega_2 \mathcal B})$  holds.
\end{Lemma}

This lemma will allow to deduce that $A_{\sigma^{\,-1} \omega_2 \mathcal B}$ is benign in $F$ from the fact that its preimage $\bar A_{\omega_2 \mathcal B}$ is benign in $\bar F$.
Namely, the direct product $\bar  K \times F$ has finite presentation 
with $31+3=34$ generators: the letters
$\bar b,\bar c;\ldots; \bar p_4$ plus the letters $a,b,c$.
Then its subgroup $Q=\big\langle
 \big(\bar w,\chi(\bar w)\big) \mathrel{|} 
 w \in   F\big\rangle$
is 3-generator as it can be generated by 
$(\bar a, a)$,
$(\bar b, b_{-1})$,
$(\bar c, c)$ already.
Hence it is benign in 
$\bar F \times F$ for the finitely presented $\bar F \times F$ and finitely generated $Q$, see Remark~\ref{RE finite generated is benign}. 
Our next objective is to modify $Q$ via a few steps to arrive to $A_{\sigma^{-1} \omega_2 \mathcal B}$ wanted.

$\bar A_{\omega_2\mathcal B} \times F$ is benign in $\bar F \times F$ for the finitely presented $\bar  K \times F$ and finitely generated $\bar L \times F$.

By Corollary~\ref{CO intersection and join are benign multi-dimensional}\;\eqref{PO 1 CO intersection and join are benign multi-dimensional} the intersection $Q_1=(\bar A_{\omega_2\mathcal B} \times F)\cap \,Q=
\big\langle
\big(\bar w,\chi(\bar w)\big) \mathrel{|} 
 w \in   A_{\omega_2\mathcal B}\big\rangle
$ is benign in $\bar F \times F$ for the finitely presented $\bigast$-construction:
$$
\mathscr{M}_1 =
\big((\bar K \times F)\;*_{\bar L \times F} w_1\big)\,*_{\bar F \times F}
\big((\bar F \times F)\;*_{Q} w_2\big)
$$
and for the finitely generated subgroup $(\bar F \times F)^{w_1 w_2}$ of $\mathscr{M}_1$. Here $w_1, w_2$ are some new letters.

Next, $\bar F \cong \bar F \times \1$ is benign in 
$\bar F \times F$ for the finitely presented $\bar F \times F$ and finitely generated $\bar F \times \1$, see Remark~\ref{RE finite generated is benign}.  Hence by the above step the join 
$
Q_2=\big\langle \bar F \times \1 ,\, Q_1 \big\rangle
=\bar F \times \langle
\chi(\bar w) \mathrel{|} 
w \in   A_{\omega_2\mathcal B}\rangle
$ is benign in $\bar F \times F$ for the finitely presented $\bigast$-construction:
$$
\mathscr{M}_2 =
\big((\bar F \times F)\,*_{\bar F \times \1} w_3\big) 
\,*_{\bar F \times F}
(\mathscr{M}_1 *_{(\bar F \times F)^{w_1 w_2}} w_4)
$$
with two next letters $w_3, w_4$,\; and for finitely generated  
$\big\langle(\bar F \times F)^{w_3}\!,\;(\bar F \times F)^{w_4} \big\rangle$ in $\mathscr{M}_2$.

Further,  
$F = \1 \times  F $ is benign in 
$\bar F \times F$ for the finitely presented $\bar F \times F$ and finitely generated $\1 \times  F $.  Hence, the intersection
$$
(\1 \times  F) \cap Q_2 
= \langle
\chi(\bar w) \mathrel{|} 
w \in   A_{\omega_2\mathcal B}\rangle
=\chi(\bar A_{\omega_2 \mathcal B})
=A_{\sigma^{-1} \omega_2 \mathcal B}
$$ is benign $\bar F \times F$ for the finitely presented $\bigast$-construction:
$$
K_{\sigma^{\,-1} \omega_2 
\mathcal B} 
= \mathscr{M}_3
=
\big((\bar F \times F)\,*_{\1 \times  F} w_5\big) 
\,*_{\bar F \times F}
(\mathscr{M}_2 *_{\langle(\bar F \times F)^{w_3},\;(\bar F \times F)^{w_4} \rangle} w_6)
$$
with two further stable letters $w_5, w_6$,\; and for finitely generated 
$L_{\sigma^{\,-1} \omega_2 \mathcal B}
=(\bar F \times F)^{w_5 w_6}$.

Hence, $A_{\sigma^{-1} \omega_2 \mathcal B}$ is benign in $F$ for the above chosen $K_{\sigma^{\,-1} \omega_2 \mathcal B}$ and   $L_{\sigma^{\,-1} \omega_2 \mathcal B}$. Although the construction looks bulky, $K_{\sigma^{\,-1} \omega_2 \mathcal B}$ has just $31+3+6=40$ generators (the generators of $\bar K \times F$ plus $6$ new stable letters $w_1, \ldots, w_6$).
And $L_{\sigma^{\,-1} \omega_2 \mathcal B} = (\bar F \times F)^{w_5 w_6}$ has $6$ generators 
$\bar a^{w_5 w_6}\!\!,\,
\bar b^{w_5 w_6}\!\!,\, 
\bar c^{w_5 w_6}\!\!;\,\;
a^{w_5 w_6}\!\!,\, 
b^{w_5 w_6}\!\!,\, 
c^{w_5 w_6}$.

\subsection{The combinatorial meaning of intersection $\iota (\omega_2\mathcal B ,\;\; \sigma^{-1}\! \omega_2 \mathcal B) $}
\label{SU Intersection} 

Let us leave aside the free constructions for a moment and understand the combinatorial meaning of intersection $\mathcal C=\omega_2\mathcal B \,\cap \,\sigma^{-1} \omega_2 \mathcal B$ of two sets $\omega_2\mathcal B$ and $\sigma^{-1} \omega_2 \mathcal B$ we discussed so far. 
In the language of Higman operations this set has the notation $\iota (\omega_2\mathcal B ,\; \sigma^{-1} \omega_2 \mathcal B)$, but due to Agreement~\ref{AG Agreement about Higman operations} it is better to describe $\mathcal C$ as the set of all sequences that can be constructed \textit{in both ways}: as concatenations of couples from $\B$ (i.e., sequences in $\omega_2\mathcal B$), and \textit{in the same time} as ``shifted'' versions of such concatenations (i.e., sequences in $\sigma^{-1}\omega_2\mathcal B$).
For example, since ${\mathcal B}=\upsilon(\zeta_{\!1} \Zz , \tau\S)$ contains the couples 
$(0,5)$, 
$(4,3)$,
$(2,1)$,
then 
$\omega_2\mathcal B$
contains the sequence 
$(0,5,\; 4,3,\; 2,1)$, and so $\sigma^{-1}\omega_2\mathcal B$ contains the ``shifted'' sequence  $\sigma^{-1} (0,5,\; 4,3,\; 2,1) = (5,4,3,2,1)$, see 
point~\ref{SU Integer functions f}, point~\ref{SU The Higman operations}, and point \ref{SU the shifted version} in the current section. 
But $(5,4,3, 2, 1) = (5,4,\; 3, 2,\; 1,0)$ also is in $\omega_2\mathcal B$ because it may be obtained as the concatenation of three couples 
$(5,4)$,\, $(3,2)$,\, $(1,0)$ from $\B$ and, hence $(5,4,3, 2, 1)\in \mathcal C$.
%

\begin{Lemma}
\label{LE Any function  from C has non-negative coordinates only}
In above notations any function $f\in \mathcal C$ has non-negative coordinates $f(i)$ only.
\end{Lemma}

\begin{proof}
$\omega_2 \mathcal B$
consists of functions $f$ constructed by blocks of type $(0,n)$ or of type $(n,n\!-\!1)$. 
For any \textit{even} $n=2,4,\ldots$ we can use the pairs $(n,n-\!1), (n\!-\!2, n\!-\!3),\ldots, (2,1)$  
to construct in $\omega_2 \mathcal B$ the sequence $f=(n,n\!-\!1,\ldots,1)$.
The sequence $g=(0,n,n\!-\!1,\ldots,1,0)$ can be built by the pair $(0,n)$ and the pairs $(n\!-\!1,n\!-\!2),\ldots, (3,2), (1,0)$.
Clearly, $\sigma^{-1} g = f$, and so $f\in  \iota(\omega_2 \mathcal B,\;\sigma^{-1} \omega_2 \mathcal B)=\mathcal C$.
In a similar manner we discover in $\mathcal C$ all the functions $f=(n,n\!-\!1,\ldots,1)$ for \textit{odd} $n=1,3,\ldots$
%
Thus, for any $n=1,2,\ldots$ the set $\mathcal C$ contains a function $f$ with the property $f(0)=n$.
This is true not only for the index $i=0$: we for any $i\in \Z$ can find an $f\in \mathcal C$ such that $f(i)=n$.

Show that  $f(i)<0$ is impossible for any $f\!\in \mathcal C$.
Assuming the contrary, suppose the \textit{least} coordinate $f(k)<0$ of $f$ is achieved at some $k$. 
If $k$ is \textit{even}, then the pair
$\big(f(k), f(k+1)\big)$ in $f\in \omega_2 \mathcal B$ has to be either of type $(n,\, n\!-\!1)$ (which is impossible as $f(k+1)$ cannot be less than $f(k)$ by minimality of $f(k)$) or it has to be of type $(0,n)$ (impossible as  $f(k)=0 \nless  0$). 
Using $\sigma^{-1}\omega_2 \mathcal B$ we in a similar way we rule out the case for an \textit{odd} index $k$.
%
\end{proof}

\subsection{The benign intersection $A_{\,\mathcal C} $} 
\label{SU The benign intersection  A_C} 
It is evident that $A_{\omega_2\mathcal B} \cap\,A_{\sigma^{-1} \omega_2 \mathcal B}=A_{\,\mathcal C}$, and then by Corollary~\ref{CO intersection and join are benign multi-dimensional}\;\eqref{PO 1 CO intersection and join are benign multi-dimensional} the group
$A_{\,\mathcal C} $ is benign in $F$ for the finitely presented group 
\begin{equation}
\label{EQ K_C}
K_{\,\mathcal C}
=
\big(K_{\omega_2 \mathcal B} *_{L_{\omega_2 \mathcal B}} x_1\big) *_F
\big(K_{\sigma^{\,-1}\omega_2 \mathcal B} *_{L_{\sigma^{\,-1}\omega_2 \mathcal B}} x_2\big),
\end{equation}
and for its finitely generated subgroup $L_{\,\mathcal C}=F^{\,x_1 x_2}$.

$K_{\,\mathcal C}$ can be given by $31+40-3+2=70$ generators: the generators of $K_{\omega_2 \mathcal B}$ and of $K_{\sigma^{\,-1}\omega_2 \mathcal B}$ (they share three generators $a,b,c$) plus two new stable letters $x_1,x_2$ from the $\bigast$-construction \eqref{EQ K_C}.
And $L_{\,\mathcal C}$ is generated by just $3$ elements $a^{x_1 x_2}\!\!,\;
 b^{x_1 x_2}\!\!,\; 
 c^{x_1 x_2}$.

\subsection{Writing $K_{\sigma^{\,-1} \omega_2 \mathcal B}$ and $K_{\mathcal C}$ by generators and defining relations}
\label{SU Writing K sigma - and K_C by generators and defining relations}

For each of $31$ generators $b,c,\ldots, q_4$ of $K_{\omega_2 \B}$ we introduced $31$ generators \textit{with bars}:
\begin{equation}
\label{EQ 31 letters with bars}
\begin{split} 
& \bar b, \bar c, \bar t_2, \bar t'_2, \bar t_0, \bar t'_0, \bar r_1, \bar r_2,
\bar g,\bar h,\bar k,\, \bar l_{0,0},\, \bar l_{1,0},\, \bar l_{1,1},
\bar s_0, \bar s_1, \bar s_2,\, \bar a, \bar r, \\
& \hskip17mm
\bar t_1, \bar t'_1, \bar u_1, \bar u_2, \bar d,\bar e, \bar v_1, \bar v_2,\;\; \bar q_1, \bar q_2, \bar q_3, \bar q_4.
\end{split}
\end{equation}
These new  $31$ letters generate the copy $\bar K_{\omega_2 \B}$ of $K_{\omega_2 \B}$, provided that we in
\eqref{EQ relations K omega2 B} also replace all letters in all relations by their copies \textit{with bars} to get new $152$ relations:
\begin{equation}
\label{EQ 152 relations with bars}
\bar b^{\bar t_2}= \bar b^{\bar c^{-1}}\!\!,\;
\bar b^{\bar t'_2}= \bar b^{\bar c^{-2}}\!\!,\;
\bar b^{\bar t_0}= \bar b^{\bar c}\!,\;
\bar b^{\bar t'_0}=\bar  b
; \; \ldots \; ;
\text{$\bar q_4$ fixes $\bar a,\bar b,\bar c$},
\end{equation}
and then write a copy of \eqref{EQ relations K omega2 B} adding bars to each letter:
\begin{equation*}
\label{EQ relations bar K omega2 B}
\begin{split}
\bar K_{\omega_2 \B}
= 
(\bar {\mathscr{L}} *_{\bar L'} \bar p_3) \;*_{\bar {\mathscr{D}}} \,
(\bar {\mathscr{D}} *_{\bar F} \bar p_4) = \big\langle
\text{$31$ generators from \eqref{EQ 31 letters with bars}}
\mathrel{\,|\,} 
\text{$152$ relations from \eqref{EQ 152 relations with bars}}
\big\rangle
\end{split}
\end{equation*}
which has again $31$ generators and $152$ relations. And $\bar L_{\omega_2 \B} = \bar {\mathscr{D}}^{\,\bar p_3 \bar p_4}$.
Then:
\begin{equation}
\label{EQ relations K omega2 B x G}
\begin{split}
\bar K_{\omega_2 \B} \!\times\! F
& = \big\langle
\text{$31$ generators from \eqref{EQ 31 letters with bars}};\; a,b,c
\mathrel{\;|\;} 
\text{$152$ relations of $\bar K_{\omega_2 \B}$ from \eqref{EQ 152 relations with bars}};\\
& \hskip11mm 
\text{$a,b,c$ commute with each of generators \eqref{EQ 31 letters with bars}}
\big\rangle.
\end{split}
\end{equation}
$\bar K_{\omega_2 \B} \times F$ has $31+3=34$ generators and $152+3\cdot 31 = 245$ relations. 

From definition of $\mathscr{M}_1$, $\mathscr{M}_2$ and $K_{\sigma^{\,-1} \omega_2 \mathcal B} = \mathscr{M}_3$ in point~\ref{SU benign A sigma omega B} we have:
\begin{equation}
\label{EQ relations M1}
\begin{split}
\mathscr{M}_1 
& = \big\langle
\text{$31$ generators from \eqref{EQ 31 letters with bars}};\; a,b,c;\; w_1, w_2
\mathrel{\;\;|\;\;} \\
& \hskip11mm 
\text{$245$ relations of $\bar K_{\omega_2 \B} \times F$ from \eqref{EQ relations K omega2 B x F}};\\
& \hskip11mm 
\text{$w_1$ fixes $19$ generators in the first row of \eqref{EQ 31 letters with bars} and $a,b,c$;}
\\
& \hskip11mm 
\text{$w_2$ fixes $\bar a a$,
$\bar b  b^{c^{-1}}$\!\!\!,\;
$\bar c c$}
\big\rangle.
\end{split}
\end{equation}
$\mathscr{M}_1$ has $31+3+2=36$ generators and $245+19+3+3 = 270$ relations. While:
\begin{equation}
\label{EQ relations M2}
\begin{split}
\mathscr{M}_2 
& = \big\langle
\text{$36$ generators from \eqref{EQ relations M1}};\; w_3, w_4
\mathrel{\;|\;} 
\text{$270$ relations of $\mathscr{M}_1$ from \eqref{EQ relations M1}}\\
& \hskip11mm 
\text{$w_3$ fixes $\bar a, \bar b, \bar c$;\;\;  $w_4$ fixes $\bar a^{w_1 w_2}, \bar b^{w_1 w_2}, \bar c^{w_1 w_2}, 
a^{w_1 w_2}, b^{w_1 w_2}, c^{w_1 w_2}$}
\big\rangle.
\end{split}
\end{equation}
$\mathscr{M}_2$ has $36+2=38$ generators and $270+3+6 = 279$ relations.  And then:
\begin{equation}
\label{EQ relations M3}
\begin{split}
K_{\sigma^{\,-1} \omega_2 \mathcal B}
& = \big\langle
\text{$38$ generators from \eqref{EQ relations M2}};\;  w_5, w_6
\mathrel{|} 
\text{$279$ relations of $\mathscr{M}_2$ from \eqref{EQ relations M2}};\\
& \hskip9mm 
\text{$w_5$ fixes $a,b,c$}; \;
\text{$w_6$ fixes conjugates of $\bar a, \bar b, \bar c,\; a,b,c$ by $w_3$ and by $w_4$}
\big\rangle.
\end{split}
\end{equation}
$K_{\sigma^{\,-1} \omega_2 
\mathcal B}$ has $38+2=40$ generators and $279+3+6\cdot 2 = 294$ relations.
And  
$L_{\sigma^{\,-1} \omega_2 \mathcal B} = (\bar F \times F)^{w_5 w_6}$ has six generators 
$\bar a^{w_5 w_6}\!\!,\,
\bar b^{w_5 w_6}\!\!,\, 
\bar c^{w_5 w_6}\!\!;\,\;
a^{w_5 w_6}\!\!,\, 
b^{w_5 w_6}\!\!,\, 
c^{w_5 w_6}$.

Finally, from definition of $K_{\,\mathcal C}$ in point~\ref{SU The benign intersection  A_C} we have:
\begin{equation}
\label{EQ relations KC}
\begin{split}
K_{\,\mathcal C} 
& = \big\langle
\text{$31$ generators of $K_{\omega_2 \B}$ from \eqref{EQ relations K omega2 B}};\\
&\hskip7.5mm
\text{$37$ generators of $K_{\sigma^{\,-1} \omega_2 
\mathcal B}$ from \eqref{EQ relations M3} (all except $a,b,c$)};\;\;   x_1, x_2
\mathrel{\;\;|\;\;}\\
& \hskip11mm 
\text{$152$ relations of $K_{\omega_2 \B}$ from \eqref{EQ relations K omega2 B}}; \;\;\;
\text{$294$ relations of $K_{\sigma^{\,-1} \omega_2 
\mathcal B}$ from \eqref{EQ relations M3}};\\
& \hskip11mm 
\text{$x_1$ fixes the conjugates of
$19$ generators of $\mathscr{D}$ from \eqref{EQ relations D}
 \,by\, $p_3 p_4$;}
\\
& \hskip11mm 
\text{$x_2$ fixes conjugates of $\bar a, \bar b, \bar c,\, a,b,c $ by $w_5 w_6$}
\big\rangle.
\end{split}
\end{equation}
Above we \textit{excluded} $a,b,c$ from the generators of $K_{\sigma^{\,-1} \omega_2 
\mathcal B}$ from \eqref{EQ relations M3} because they already were included among the $31$ generators of $K_{\omega_2 \B}$ from \eqref{EQ relations K omega2 B}.
$K_{\,\mathcal C}$ has $31+40 - 3 +2=70$ generators and $152+294+19+6 = 471$ relations.
$L_{\,\mathcal C}=F^{\,x_1 x_2}$ is generated by just $3$ elements $a^{x_1 x_2}\!\!,\;
 b^{x_1 x_2}\!\!,\; 
 c^{x_1 x_2}$\!.

\section{Construction of $K_{\mathcal F}$ and $K_{\mathcal T}$
}
\label{SE Construction of KF and KT}

\subsection{Functions $f$ with the condition $f(4) = f(5)-2$} 
\label{SU Functions with the condition f(4) = f(5)-2}

Denote by $\mathcal D$ the set \textit{all} functions  $f \in \E$  such that $f(4) = f(5)-2$,
and $f(i)=0$ for all $i \neq 4,5$,  i.e., the set of all sequences of type $(0, 0, 0, 0, 0,\; n\!-\!2,\; n)$ with all $n\in \Z$.
It is clear that in $F$ we have  
$
a^{b_5^2}=a^{b_0^0 b_1^0 b_2^0 b_3^0 b_4^0 b_5^2 b_6^0}=a_{(0, 0, 0, 0, 0, 2, 0)}
$
for the sequence $f=(0, 0, 0, 0, 0, 2, 0)\in \mathcal D$.
Returning to the group $\mathscr{A}$ of point~\ref{SU Building the group script A}, again use Lemma~\ref{LE action of d_m on f} in the manner used in point~\ref{SU KB is benign}, to get 
$
F \;\cap\; \langle
a^{b_5^2},\;\, d_5 d_6
\rangle = A_{\mathcal D}
$. 
Hence, the subgroup $A_{\mathcal D}$ is benign in $F$ for the finitely presented overgroup $K_{\mathcal D}\! = \!\mathscr{A}$ of $F$, and for the finitely generated subgroup $L_{\mathcal D} = \langle
a^{b_5^2},\; d_5 d_6
\rangle$ of $K_{\mathcal D}$.

\subsection{Intersection ${\mathcal C}\cap {\mathcal D}$} 
\label{SU The intersection cap D}

The combinatorial meaning of the intersection ${\mathcal C}\cap {\mathcal D}$ is easy to see: since ${\mathcal C}$ consists of all sequences $f \in {\mathcal E}$ with \textit{non-negative} coordinates, see Lemma~\ref{LE Any function  from C has non-negative coordinates only}, and since the sequences $f \in {\mathcal D}$ satisfy $f(4) = f(5)-2$, then 
${\mathcal C}\cap {\mathcal D}$ is noting but the set of all sequences  $(0, 0, 0, 0, 0, n\!-\!2, n)$ with $n=2,3,4,\ldots$ only. 

In points \ref{SU The benign intersection  A_C} and \ref{SU Functions with the condition f(4) = f(5)-2} we saw that $A_{\mathcal C}$ and $A_{\mathcal D}$ are benign in $F$ for these two specific subsets ${\mathcal C}$ and ${\mathcal D}$ of ${\mathcal E}$.  
Hence $A_{\,{\mathcal C}\cap {\mathcal D}}$ also is benign in $F$.
Using Corollary~\ref{CO intersection and join are benign multi-dimensional}\;\eqref{PO 1 CO intersection and join are benign multi-dimensional} together with the respective groups $K_{\,\mathcal C}$ and $K_{\mathcal D}=\mathscr{A}$ along with their finitely presented subgroups $L_{\mathcal C}$ and $L_{\mathcal D}$ we get that in finitely presented $\bigast$-construction:
\begin{equation}
\label{EQ K_C cap D}
K_{\,{\mathcal C}\cap\, {\mathcal D}}
=
\big(K_{\mathcal C} *_{L_{\mathcal C}} y_1\big) *_{\mathscr{A}}
\big(\mathscr{A} *_{L_{\mathcal D}} y_2\big)
\end{equation}
we have $F \cap F^{y_1 y_2} = A_{\,{\mathcal C}\cap {\mathcal D}}$. So it remains to choose the $3$-generator subgroup $L_{\,{\mathcal C}\cap {\mathcal D}} =  F^{y_1 y_2}$ to have that $A_{\,{\mathcal C}\cap {\mathcal D}}$ is benign.

\begin{Remark}
Comparing \eqref{EQ K_C cap D} with the initial definition in 
\eqref{EQ initial form of star construction}
and
\eqref{EQ nested Theta  multi-dimensional short form} 
notice that here we took $M=\mathscr{A}$ and not $M=F$ because $K_{\mathcal C} *_{L_{\mathcal C}} y_1$ and $\mathscr{A} *_{L_{\mathcal D}} y_2$ intersect in $\mathscr{A}$ (the elements are $b,c,t_1, t'_1, a, u, v, d, e$ shared by $K_{\mathcal C}$ and $\mathscr{A}$, and so choosing a smaller amalgamated subgroup $F$ we would have to deal with \textit{distinct} copies of, say, $u,v$ in $K_{\mathcal C}$ and in $K_{\mathcal D}=\mathscr{A}$). 
\end{Remark}

\subsection{Extending the construction for the set $\mathcal F$} 
\label{SU Extending the construction for the set F}

Let us enlarge the set of ${\mathcal C}\cap {\mathcal D}$ of previous point to the subset $\mathcal F$ of $\mathcal E$ consisting of \textit{all} sequences:  
\begin{equation}
\label{EQ sequences in F}
f=\big(j_0,\, j_1,\, \ldots,\, j_{5},\, \ldots,\, j_{18} \big) \in \mathcal E
\end{equation}
such that $f(5) = j_{5} \ge 2$ and each 
$f(i) = j_{i}$ is an \textit{arbitrary} integer for $i=0, \ldots, 4$ and for $i=6, \ldots, 18$. 
Clearly ${\mathcal C}\cap {\mathcal D}$ is a subset of $\mathcal F$.
Let us show that $A_{\mathcal F}$ is benign in $F$.

Denote $K_{\mathcal F}=K_{\,{\mathcal C}\cap {\mathcal D}}$  and 
\begin{equation}
\label{EQ the subgroup LF}
\begin{split}
L_{\mathcal F} & = \big\langle L_{\,{\mathcal C}\cap {\mathcal D}},\;\; d_i \mathrel{|}  i=0, \ldots, 4,\; 6, \ldots, 18 \big\rangle \\
& = \big\langle a^{y_1 y_2}\!,\; b^{y_1 y_2}\!,\; c^{y_1 y_2}\!,\; \; d_i \mathrel{|}  i=0, \ldots, 4,\; 6, \ldots, 18 \big\rangle
\end{split}
\end{equation}
a $21$-generator subgroup in $K_{\mathcal F}$. 
It is easy to calculate that for any sequence $f\in \mathcal F$ the element $a_f$ is in $F \cap L_{\mathcal F}$. Say, for 
$f=(3, 
1, 0, 2, 1, \, \boldsymbol 8, \,
4, 6, -5, 0, 1, 
2, 1, 3,-4, 2, 
1, 5, -7)$
(notice that the $5$'th coordinate $j_{5}=\boldsymbol 8$ is greater than or equal to $2$) we can take the sequence 
$g=(0, 
0, 0, 0, 0, \, \boldsymbol 8, \,
\boldsymbol 7)\, \in\, {\mathcal C}\cap {\mathcal D}$,
the product 
\begin{equation}
\label{EQ bold u}
d' =
d_0^{3}\, 
d_1^{}\, d_2^{0}\, d_3^{2}\, d_4^{} \;\; 
\boldsymbol{d_6^{-3}}\; d_7^{6}\, d_8^{-5} \,d_9^{0}\, d_{10}^{} \,
d_{11}^{2}\, d_{12}^{}\, d_{13}^{3} \,d_{14}^{-4} \,d_{15}^{2} \,
d_{16}^{} \,d_{17}^{5}\, d_{18}^{-7} 
\end{equation}
(notice how above $d_5^{}$ is missing, and how we used ${d_6^{-3}}$),
and calculate:
\begin{equation}
\label{EQ a_f for long f}
\begin{split}
\big(a_g\big)^{d'}\!\!\!\!  = \big(a^{b_g}\big)^{d'}  
 \!\!\!\!= \big(a^{b_5^{\boldsymbol 8} \, b_6^{\boldsymbol 7}}\big)^{d'} \!\!\!\! = a^{
b_0^{3}\,
b_1^{}\,
b_2^{0}\,
b_3^{2}\,
b_4^{}\,
\cdot
\,
b_5^{\boldsymbol 8} \; b_6^{\boldsymbol 7-3}
\cdot
\;
b_7^{6}\,
b_8^{-5}\,
b_9^{0}\,
b_{10}^{}\,
b_{11}^{2}\,
b_{12}^{}\,
b_{13}^{3}\,
b_{14}^{-4}\,
b_{15}^{2}\,
b_{16}^{}\,
b_{17}^{5}\,
b_{18}^{-7}
}
=a^{b_f} \! = a_f
\end{split}
\end{equation}
(notice that $b_6^{7-3}\!=\,b_6^{4}$). And since also $a_f \in F$, we have $A_{\mathcal F} \subseteq F \cap L_{\mathcal F}$.

\medskip
To show the reverse inclusion apply the ``conjugates collecting'' process of point~\ref{SU The conjugates collecting process} with $\X=L_{\,{\mathcal C}\cap {\mathcal D}}$ and $\Y=\big\{ d_i \mathrel{|}  i=0, \ldots, 4,\; 6, \ldots, 18 \big\}$.
Any $l\in \langle \X,\Y \rangle = L_{\mathcal F}$ can be written as:
\begin{equation}
\label{EQ conjugates collecting for l}
l
= x_1^{\pm v_1}
x_2^{\pm v_2}
\cdots
x_{k}^{\pm v_k}
\cdot
v
\end{equation}
where  $v_1,v_2,\ldots,v_k,\, v\in\langle \Y \rangle$, and where
$x_1,x_2,\ldots,x_k \in \X = L_{\,{\mathcal C}\cap {\mathcal D}}$. 

If all elements $x_i$ happen to be in $F$, then they also are in the intersection $F \cap L_{\,{\mathcal C}\cap {\mathcal D}}$ which coincides to $A_{\,{\mathcal C}\cap {\mathcal D}}$. If so, then regardless which words $v_i\in \Y$ we pick, the conjugates $x_i^{\pm v_i}$ are being calculated exactly in the way we did it in \eqref{EQ a_f for long f}, and so  $x_i^{\pm v_i} \!\! \in A_{\mathcal F}$. Also,  $l\in F$ necessarily implies 
$v=1$, that is $l\in A_{\mathcal F}$.
On the other hand, if there are some elements $x_i$ in 
\eqref{EQ conjugates collecting for l} are outside $F$, the their conjugates $x_i^{\pm v_i}$ also are outside $F$. I.e., in such a case \eqref{EQ conjugates collecting for l} may be in $F$ only if all such $x_i^{\pm v_i}$ cancel each other out, plus $v=1$, of course.

Hence, $A_{\mathcal F}$ is benign in $F$ for the above mentioned $K_{\mathcal F}=K_{\,{\mathcal C}\cap {\mathcal D}}$  and $L_{\mathcal F}$ from \eqref{EQ the subgroup LF}.


\subsection{Defining a larger set $\mathcal H$ covering $\mathcal T$}
\label{SE Defining a larger set}

Following the notation of \eqref{EQ Higman code for Q}, for $k\!=\!2$ pick:
\begin{equation}
\label{EQ f_2}
f_2=
\big(1,\;\boldsymbol{-2},\;-1, \;\boldsymbol{-2} ,\; -1,\;\; \boldsymbol{2},\; \; 1,\; \; \boldsymbol{2},\; \; 1,\;  1, -1,\;\;  \boldsymbol{-1},\; -1,-1,\,  1,\;\;  \boldsymbol{1},\;\;  1,\;\;  \boldsymbol{1},\;\;  -1\big).
\end{equation}
for convenience of later use we stressed some of the coordinates  above in bold (those coordinates which actually depend on $k$).
Then in $F$ and in $\mathscr{A}$ define:
\begin{equation}
\label{EQ a f 2}
a_{f_2}=a^{b_{f_2}},
\end{equation}
\begin{equation}
\label{EQ d''}
d'' = d_1^{-1} d_3^{-1} d_5\, d_7\, d_{11}^{-1} d_{15}\, d_{17}.
\end{equation}
It is not hard to verify  that $a_{f_2}^{d''^{\,n}}=a_{f_{2+n}}$ holds for any $n\in \Z$. Indeed, using Lemma~\ref{LE action of d_m on f}, in analogy with say \eqref{EQ action of d1} or \eqref{EQ action of d0d1}, we for any \textit{positive}  $n=1,2,\ldots$ have:
\begin{equation}
\label{EQ expand in positive direction}
\begin{split}
a_{f_2}^{d''^{\,n}}
&=\big(a_{f_2}^{d''}\big)^{{d''}^{\,n-1}}
\!\!\!
=\Big(\big(a^{b_{f_2}}\big)^{d''}\Big)^{d''^{\,n-1}}
\\
&=\Big(\big(a^{
b_0\, b_1^{-2}\, b_2^{-1}\, b_3^{-2}\, b_4^{-1}\, b_5^{2}\; b_6\; b_7^{2}\; b_8\, b_9\, b_{10}^{-1}\, b_{11}^{-1}\, b_{12}^{-1}\, b_{13}^{-1}\, b_{14}\, b_{15}\, b_{16}\, b_{17}\, b_{18}^{-1}
}\big)^{d_1^{-1} d_3^{-1} d_5\, d_7\, d_{11}^{-1} d_{15}\, d_{17}}\Big)^{{d''}^{\,n-1}}\\
&=\Big(a^{
b_0\, b_1^{-3}\, b_2^{-1}\, b_3^{-3}\, b_4^{-1}\, b_5^{3}\; b_6\; b_7^{3}\; b_8\, b_9\, b_{10}^{-1}\, b_{11}^{-2}\, b_{12}^{-1}\, b_{13}^{-1}\, b_{14}\, b_{15}^2\, b_{16}\, b_{17}^2\, b_{18}^{-1}
}\Big)^{{d''}^{\,n-1}}\\
&=\;a_{f_3}^{{d''}^{\,n-1}}
\!\!\!
=\; a_{f_4}^{{d''}^{\,n-2}}
\!\!\!
= \;\cdots \; = 
\; a_{f_{1+n}}^{d''}
\!\!\!
=\; a_{f_{2+n}}.
\end{split}
\end{equation}
The analog of this  for  \textit{negative}  powers is equally easy to verify:
\begin{equation}
\label{EQ expand in negative direction}
a_{f_2}^{{d''}^{-n}} 
= \big(a_{f_2}^{{d''}^{-1}}\big)^{{d''}^{-(n-1)}}
\!\!\!
=\; a_{f_{2-n}}.
\end{equation}
Denote by $\mathcal H$ the set of all sequences of type:
\begin{equation*}
\label{EQ set larger than F}
\big(1,-2\!-\!n,-1, -2\!-\!n , -1,\; 2+n,\;  1,\;  2+n, \; 1,\;  1, -1,\;  1\!\!-\!n, -1,-1,\;  1,\; n\!-\!\!1,\;  1,\;  n\!-\!\!1,\;  -1\big)
\end{equation*}
for \textit{all} integers $n\in \Z$.
This set $\mathcal H$ clearly contains the set $\mathcal T$ given earlier by \eqref{EQ Higman code for Q}. 

The above equalities 
\eqref{EQ expand in positive direction} and 
\eqref{EQ expand in negative direction} together mean that $\mathcal H$ is benign in $F$, as $F \cap L_{\mathcal H} = A_{\mathcal H}$ holds in the finitely pretended overgroup $K_{\mathcal H}=\mathscr{A}$ of $F$, and for the finitely generated subgroup $L_{\mathcal H}=\langle a_{f_2}, \; {d''}\rangle$ of the latter.

\subsection{The intersection $\mathcal F \cap \mathcal H = \mathcal T$}
\label{SU Obtaining a larger set}

The set $\mathcal T$ of Higman codes for $T_\Q$ defined in \eqref{EQ Higman code for Q} is nothing but the set of all those sequences from $\mathcal H$ for which $f(5) \ge 2$ holds.

On the other hand the set $\mathcal F$ of sequences \eqref{EQ sequences in F} constructed in point~\ref{SU Extending the construction for the set F} consists of all $f\in \mathcal E$
for which $f(5) \ge 2$, the coordinates 
$f(i)$ may be \textit{any} integers for $i=0, \ldots, 4$ and for $i=6, \ldots, 18$, and the coordinates $f(i)$ all are zero for any $i\notin \{ 0, 1, \ldots, 18 \}$.
Hence, 
$
\mathcal F \cap \mathcal H = \mathcal T,
$
and since $A_{\mathcal F},\,A_{\mathcal H}$ both are benign in $F$, the intersection $A_{\mathcal T}=A_{\mathcal F} \cap A_{\mathcal H}$ also is benign in $F$. Corollary~\ref{CO intersection and join are benign multi-dimensional}\;\eqref{PO 1 CO intersection and join are benign multi-dimensional} lets us construct the  finitely presented overgroup $K_{\mathcal T}$ of $F$:
\begin{equation*}
\label{EQ KT}
K_{\mathcal T}
=
\big(K_{\mathcal F} *_{L_{\mathcal F}} z_1\big) *_{\mathscr{A}}
\big(K_{\mathcal H} *_{L_{\mathcal H}} z_2\big)
\end{equation*}
(in this $\bigast$-construction $M=K_{\mathcal H}=\mathscr{A}$),
and its finitely generated subgroup $
L_{\mathcal T}=F^{z_1 z_2}$.

\smallskip
The goal set at the end of  point~\ref{SU Getting the  Higman code from TQ} has been achieved.

\subsection{Writing $K_{\mathcal F}$ and $K_{\mathcal T}$  by generators and defining relations}
\label{SU Writing K_F and K_T by generators and defining relations} 

Using $K_{\,\mathcal C}$ in point~\ref{SU The benign intersection  A_C}, 
$L_{\mathcal F}$ in point~\ref{SU Extending the construction for the set F}, 
$K_{\mathcal F}=K_{\,{\mathcal C}\cap {\mathcal D}}$ in point~\ref{SU The intersection cap D} and point~\ref{SU Extending the construction for the set F},
$L_{\,\mathcal C}$ in point~\ref{SU The benign intersection  A_C} and 
$L_{\mathcal D}$ in point~\ref{SU Functions with the condition f(4) = f(5)-2}:
\begin{equation}
\label{EQ relations KF}
\begin{split}
K_{\,\mathcal F} 
& = \big\langle
\text{$70$ generators of $K_{\mathcal C}$ from \eqref{EQ relations KC}};\;\;\;  y_1, y_2
\mathrel{\;\;|\;\;}\\
& \hskip11mm 
\text{$471$ relations of $K_{\mathcal C}$ from \eqref{EQ relations KC}};\\[-5pt]
& \hskip11mm 
\text{$y_1$ fixes $a^{x_1 x_2}\!\!,\;
 b^{x_1 x_2}\!\!,\; 
 c^{x_1 x_2}$}; \;\;\;\;\;
\text{$y_2$ fixes $a^{(b^2)^{c^5}}\!\!\!,\;\; d^{e^5}\! d^{e^6}$}
\big\rangle.
\end{split}
\end{equation}
Here $K_{\,\mathcal F}$ has $70 +2=72$ generators and $471+5 = 476$ relations.

And 
$L_{\mathcal F} = \big\langle a^{y_1 y_2}\!,\; b^{y_1 y_2}\!,\; c^{y_1 y_2}\!,\; \; d_i \text{\, with \,}  i=0, \ldots, 4,\; 6, \ldots, 18 \big\rangle$ has
$3+18=21$ generators.

Finally, using 
$K_{\mathcal H}=\mathscr{A}$ and
$L_{\mathcal H}$,\; 
$d''$,\;
$a_{f_2}$ 
from point~\ref{SE Defining a larger set}
together with 
$\mathcal F \cap \mathcal H = \mathcal T$, \;
$K_{\mathcal T}$ and
$L_{\mathcal T}$
from point~\ref{EQ expand in negative direction} we write:
\begin{equation}
\label{EQ relations KТ}
\begin{split}
K_{\,\mathcal T} 
& = \big\langle
\text{$72$ generators of $K_{\mathcal F}$ from \eqref{EQ relations KF}};\;\;\;  z_1, z_2
\mathrel{\;\;|\;\;}\\
& \hskip11mm 
\text{$476$ relations of $K_{\mathcal F}$ from \eqref{EQ relations KF}};\\[-5pt]
& \hskip11mm 
\text{$z_1$ fixes $a^{y_1 y_2}\!,\; b^{y_1 y_2}\!,\; c^{y_1 y_2}$ and all $d^{e^i}$ for $i=0, \ldots, 4,\; 6, \ldots, 18$};\\[-5pt]
& \hskip11mm 
\text{$z_2$ fixes the elements $a_{f_2}$ of \eqref{EQ a f 2}, and $d''$ of \eqref{EQ d''}}
\big\rangle.
\end{split}
\end{equation} 
Here 
$K_{\,\mathcal T}$ has $72 +2=74$ generators and $476+3+18+2= 499$ relations.
And 
$L_{\mathcal T}=F^{z_1 z_2} 
= \big\langle a^{z_1 z_2}\!,\; b^{z_1 z_2}\!,\; c^{z_1 z_2} \big\rangle$

\section{The final embedding step with the Higman Rope Trick}
\label{SE The final embedding}

\noindent
In first two points of this section we show that if $A_{\mathcal T}$ is benign in $F$, then a respective subgroup $W_{\mathcal T}$ is benign in the free group $\langle p,q \rangle$ of rank $2$. Then in point~\ref{SU The Higman Rope Trick} we apply the ``Higman Rope Trick'' (see \cite{Higman rope trick}) to build the finitely presented overgroup $\mathcal Q$ of $\Q$.

\subsection{If $A_{\mathcal T}$ is benign in $\langle a, b, c\rangle$, then $Z_{\mathcal T}$ is benign in $\langle z, m, n\rangle$}
\label{SU If AB is benign then ZB also is benign}

For this section we need some notation slightly differing from  $b_f$ and $a_f$ we used so far in $F=\langle a,b,c \rangle$.
Namely, in the free group $\langle m,n \rangle\cong F_2$ for any sequence $f\in \E$ (in particular, for any $f \in \mathcal T$) introduce the word 
$w_f(m,n) = m^{f(0)} n^{f(1)} \cdots m^{f(2r)} n^{f(2r\!+\!1)}$.
For example, if $f=(3,5,2)=(3,5,2,0)$, then $w_f(m,n) = m^{3}n^{5}m^{2} = m^{3}n^{5}m^{2}n^{0}$.
Then inside a larger free group $\langle z,m,n \rangle\cong F_3$ denote 
$Z_{\mathcal T}= \big\langle z^{w_f(m,n)} \mathrel{|} f \in \mathcal T \big\rangle$.
In a similar way, in $\langle p,q \rangle \cong F_2$ introduce the word 
$w_f(p,q) = p^{f(0)} q^{f(1)} \cdots p^{f(2r)} q^{f(2r\!+\!1)}$\!,\, and define the subgroup
$W_{\mathcal T}  = 
\langle w_f(p,q) \mathrel{|} f \in \mathcal T \rangle$.

\medskip
The free group $\langle z, m, n\rangle$ admits an isomorphism $\mu$ sending $z, m, n$ to $z, n, m$ (it just swaps $m$ and $n$), and using it we build the HNN-extension $\mathscr{K}= \langle z, m, n\rangle *_{\mu} \kappa$. 
Since $F=\langle a, b, c\rangle$ is free, there is an injection $\phi:F \to \mathscr{K}$ sending $a, b, c$ to $z, m, \kappa$. It is very easy to see that for each $f$ we have $\phi(a_f) = z^{w_f(m,n)}$\!.\, For instance, if $f=(3,5,4,7)$,  we have:
$$
a_f= a^{b_f}=
a^{
\big(b^{(c^0)}\big)^{\! 3}
\big(b^{(c^1)}\big)^{\! 5}
\big(b^{(c^2)}\big)^{\! 4}
\big(b^{(c^3)}\big)^{\! 7}
}
\xrightarrow{\phantom{-} \phi \phantom{-}}
\;
z^{
\big(m^{(\kappa^{\,0})}\big)^{\! 3}
\big(m^{(\kappa^{\,1})}\big)^{\! 5}
\big(m^{(\kappa^{\,2})}\big)^{\! 4}
\big(m^{(\kappa^{\,3})}\big)^{\! 7}
}
\!\!\!\!
=z^{
m^3 n^5 m^4 n^7
}
$$
because
$\big(m^{(\kappa^{\,0})}\big)^{\! 3}
\!\!=\!\big(\mu^0(m) \big)^{\! 3}\!=\!m^3$,\,
$\big(m^{(\kappa^{\,1})}\big)^{\! 7}
\!\!=\!\big(\mu(m)\big)^{\! 7}\!=\!n^7$,\,
$\big(m^{(\kappa^{\,2})}\big)^{\! 4}
\!\!=\!\big(\mu^2(m) \big)^{\! 4}\!=\!m^4$
and
$\big(m^{(\kappa^{\,3})}\big)^{\! 7}
\!\!=\!\big(\mu^3(m)\big)^{\! 7}\!=\!n^7$.
Hence, $Z_{\mathcal T}$ is the image of 
$A_{\mathcal T}$ under $\phi$, and we can use a trick similar to what we used 
in point~\ref{SU benign A sigma omega B}. For now we just assume the finitely presented group $K_{\mathcal T}$ and its finitely generated subgroup $L_{\mathcal T}$ are \textit{known} for $\mathcal T$ (without using their \textit{actual structure} found in point~\ref{SU Obtaining a larger set}).
The direct product $K_{\mathcal T} \times \mathscr{K}$ is finitely presented, and its subgroup $Q=\big\langle
 \big(w,\phi(w)\big) \mathrel{|} 
 w \in   F\big\rangle$
is just 3-generator.
Hence it is benign in 
$F \times \mathscr{K}$ for the finitely presented $F \times \mathscr{K}$ and finitely generated $Q$. 
In analogy to point~\ref{SU benign A sigma omega B} we modify $Q$ via a few steps to arrive to $Z_{\mathcal T}$ wanted.

$A_{\mathcal T} \times \mathscr{K}$ is benign in $F \times \mathscr{K}$ for the finitely presented $K_{\mathcal T} \times \mathscr{K}$ and finitely generated $L_{\mathcal T} \times \mathscr{K}$.
Hence by Corollary~\ref{CO intersection and join are benign multi-dimensional}\;\eqref{PO 1 CO intersection and join are benign multi-dimensional} the intersection $Q_1=(A_{\mathcal T} \times \mathscr{K})\cap Q=
\big\langle
\big(w,\phi(w)\big) \mathrel{|} 
 w \in   A_{\mathcal T}\big\rangle
$ is benign in $F \times \mathscr{K}$ for the finitely presented $\bigast$-construction:
$$
K_{Q_1} =
\big((K_{\mathcal T} \times \mathscr{K})\;*_{L_{\mathcal T} \times \mathscr{K}} q_1\big)\,
*_{F \times \mathscr{K}}
\big((F \times \mathscr{K})\;*_{Q} q_2\big)
$$
and for the finitely generated subgroup $L_{Q_1} =(F \times \mathscr{K})^{q_1 q_2}$ of the latter. 

Next, $F \cong F \times \1$ is benign in 
$F \times \mathscr{K}$ for the finitely presented $F \times \mathscr{K}$ and finitely generated $F \times \1$.  Hence the join 
$
Q_2=\big\langle F \times \1 ,\, Q_1 \big\rangle
=F \times \langle
\phi(w) \mathrel{|} 
w \in   A_{\mathcal T}\rangle
$ is benign in $F \times \mathscr{K}$ for the finitely presented $\bigast$-construction:
$$
K_{Q_2}=
\big((F \times \mathscr{K})\,*_{F \times \1} q_3\big) 
\,*_{F \times \mathscr{K}}
(K_{Q_1} *_{(F \times \mathscr{K})^{q_1 q_2}} \;q_4),
$$
and for the finitely generated subgroup 
$L_{Q_2} =\big\langle(F \times \mathscr{K})^{q_3},\;(F \times \mathscr{K})^{q_4} \big\rangle$ in $K_{Q_2}$.

Further,  
$\mathscr{K} = \1 \times  \mathscr{K}$ is benign in 
$F \times \mathscr{K}$ for the finitely presented $F \times \mathscr{K}$ and finitely generated $\1 \times  \mathscr{K}$.  Hence, the intersection
$$
(\1 \times  \mathscr{K}) \cap Q_2 
= \langle
\phi(w) \mathrel{|} 
w \in   A_{\mathcal T}\rangle
=\phi(A_{\mathcal T})
=Z_{\mathcal T}
$$ 
is benign in $F \times \mathscr{K}$ for the finitely presented $\bigast$-construction:
$$
K_{Z_{\mathcal T}} =
K_{\phi (\mathcal T)} =
\big((F \times \mathscr{K})\,*_{\1 \times  \mathscr{K}} q_5\big) 
\,*_{F \times \mathscr{K} }
(K_{Q_2} *_{\langle(F \times \mathscr{K})^{q_3},\;(F \times \mathscr{K})^{q_4} \rangle} q_6),
$$
and for finitely generated subgroup
$L_{Z_{\mathcal T}} =L_{\phi (\mathcal B)}
=(F \times \mathscr{K})^{q_5 q_6}$.
Hence, $Z_{\mathcal T}=\phi(A_{\mathcal T})$ is benign in $\mathscr{K}$ for the above chosen $K_{Z_{\mathcal T}}$ and $L_{Z_{\mathcal T}}$. 
But since $Z_{\mathcal T}$ is inside $\langle z,m,n \rangle$, 
we have $\langle z,m,n \rangle \cap L_{Z_{\mathcal T}} = Z_{\mathcal T}$, that is $Z_{\mathcal T}$ also is benign in  $\langle z,m,n \rangle$ for the same $K_{Z_{\mathcal T}}$ and $L_{Z_{\mathcal T}}$.

\subsection{Writing $K_{Z_{\mathcal T}}$  by generators and defining relations}
\label{SU Writing KZT by generators and defining relations} 

Using definitions of $\mathscr{K}$, $Q$, $\varphi$, $K_{Q_1}$ in point~\ref{SU If AB is benign then ZB also is benign} and of 
$L_{\mathcal T}$ in point~\ref{SU Writing K_F and K_T by generators and defining relations} we have:
\begin{equation}
\label{EQ relations Q1}
\begin{split}
K_{Q_1} 
& = \big\langle
\text{$74$ generators of $K_{\mathcal T}$ from \eqref{EQ relations KТ}};\;\;\;  z, m, n, \kappa; \;\;  q_1, q_2
\mathrel{\;\;|\;\;}\\
& \hskip11mm 
\text{$499$ relations of $K_{\mathcal T}$ from \eqref{EQ relations KТ}};
\;\;\; 
\text{$\kappa$ sends $z, m, n$ to $z, n, m$};\\[-2pt]
& \hskip11mm 
\text{each of $z, m, n, \kappa$ commutes with $74$ generators of $K_{\mathcal T}$ from \eqref{EQ relations KТ}};\\[-2pt]
& \hskip11mm 
\text{$q_1$ fixes $a^{z_1 z_2}\!,\; b^{z_1 z_2}\!,\; c^{z_1 z_2}$ and $z, n, m, \kappa$}; \;\;\;\; 
\text{$q_2$ fixes $az,\, bm,\, c \kappa$}
\big\rangle.
\end{split}
\end{equation} 
$K_{Q_1}$ has $74+4+2=80$ generators and $499+3+4\cdot 74 +7 +3 = 808$ relations.
And 
$L_{Q_1} =(F \times \mathscr{K})^{q_1 q_2}$ is a $7$-generator group. 

Next, using definition of $K_{Q_2}$ in point~\ref{SU If AB is benign then ZB also is benign} we write:
\begin{equation}
\label{EQ relations Q2}
\begin{split}
K_{Q_2} 
& = \big\langle
\text{$80$ generators of $K_{Q_1}$ from \eqref{EQ relations Q1}};\;\;\;  q_3, q_4
\mathrel{\;\;|\;\;}\\
& \hskip11mm 
\text{$808$ relations of $K_{Q_1}$ from \eqref{EQ relations Q1}}; \;\;\;
\text{$q_3$ fixes $a,b,c$};\\[-2pt]
& \hskip11mm 
\text{$q_4$ fixes conjugates of $a,b,c;\;  z, m, n, \kappa$ by $q_1 q_2$}
\big\rangle.
\end{split}
\end{equation}
$K_{Q_2}$ has $80+2=82$ generators and $808+3+7 = 818$ relations. 
$L_{Q_2} =\big\langle(F \times \mathscr{K})^{q_3},\;(F \times \mathscr{K})^{q_4} \big\rangle$  is a $14$-generator group.

Finally, by definition of $K_{Z_{\mathcal T}}$
again in point~\ref{SU If AB is benign then ZB also is benign} we have:
\begin{equation}
\label{EQ relations KZT}
\begin{split}
K_{Z_{\mathcal T}} 
& = \big\langle
\text{$82$ generators of $K_{Q_2}$ from \eqref{EQ relations Q2}};\;\;\;  q_5, q_6
\mathrel{\;\;|\;\;}\\
& \hskip11mm 
\text{$818$ relations of $K_{Q_2}$ from \eqref{EQ relations Q2}}; \;\;\;
\text{$q_5$ fixes $z, m, n, \kappa$};\\[-2pt]
& \hskip11mm 
\text{$q_6$ fixes conjugates of $a,b,c;\;  z, m, n, \kappa$ by $q_3$ and by $q_4$}
\big\rangle.
\end{split}
\end{equation}
$K_{Z_{\mathcal T}}$ has $82+2=84$ generators and $818+4+2\cdot 7 = 836$ relations.
And 
$L_{Z_{\mathcal T}} =L_{\phi (\mathcal B)}
=(F \times \mathscr{K})^{q_5 q_6}$
 is a $14$-generator group.

\subsection{If $Z_{\mathcal T}$ is benign in $\langle z, m, n\rangle$, then $W_{\mathcal T}$ is benign in $\langle p,q \rangle$}
\label{SU If ZT is benign then W(p,q) also is benign}

Next suppose $Z_{\mathcal T}$ is  benign in  $\langle z,m,n \rangle$ for some finitely presented $K_{Z_{\mathcal T}}$ and its finitely generated subgroup $L_{Z_{\mathcal T}}$, in particular, these can be the groups from previous point.
Let us show that the subgroup $W_{\mathcal T} = 
\langle w_f(p,q) \mathrel{|} f \in \mathcal T \rangle$ is benign in the free group $\langle p,q \rangle \cong F_2$. 

For an infinite cycle $\langle u \rangle$ the free product $K_{Z_{\mathcal T}} * \langle u \rangle$ contains the free subgroup $\langle z,m,n \rangle * \langle u \rangle=\langle z,m,n, u \rangle \cong F_4$. From the normal form of elements in free products it is trivial that in $K_{Z_{\mathcal T}}\! * \langle u \rangle$ the equality $\langle z,m,n, u \rangle \cap L_{Z_{\mathcal T}}\!=Z_{\mathcal T}$ holds, i.e., $Z_{\mathcal T}$ is benign in $\langle z,m,n, u \rangle$ for the finitely presented $K_{Z_{\mathcal T}}\! * \langle u \rangle$ and its finitely generated subgroup $L_{Z_{\mathcal T}}$.

In $\langle z,m,n, u \rangle$ the words $u^{z^{w(m,n)}}$ freely generate subgroup of countable rank. Here $w(m,n)$ run through the set of \textit{all possible} words on $m,n$ (we do not \textit{yet} put any restriction depending on $\mathcal T$).
For another infinite cycle $\langle v \rangle$ in the free product $\langle p,q\rangle * \langle v \rangle=\langle p,q,v \rangle \cong F_3$ select the countable many free generators $v\cdot w(p,q)$. 
The map $\theta$ sending each $u^{z^{w(m,n)}}\!$ to $v\cdot w(p,q)$ can be continued to an isomorphism between two isomorphic (free) subgroups of $K_{Z_{\mathcal T}}\! * \langle u \rangle$ and $\langle p,q,v \rangle$, and this lets us define the following free products with amalgamation:
$$
\bar{\mathscr{P}}
=
\big( K_{Z_{\mathcal T}}*\langle u \rangle \big)
*_{\theta}
\langle p,q,v \rangle
,\quad\quad\quad
\mathscr{P}
=
\langle z,m,n,u\rangle 
*_{\theta}
\langle p,q,v \rangle
$$  
where $\mathscr{P}\! \le\! \bar{\mathscr{P}}$.
Since the words $v\!\cdot\! w(p,q)$ generate  $\langle p,q,v \rangle$, the latter is inside $K_{Z_{\mathcal T}}\!\!*\langle u \rangle$, and we in fact have $\bar{\mathscr{P}}
\cong
K_{Z_{\mathcal T}}*\langle u \rangle$ and $\mathscr{P}
\cong
\langle z,m,n,u\rangle$, that is, $\bar{\mathscr{P}}$ can be understood as the group $K_{Z_{\mathcal T}}*\langle u \rangle$ in which we denoted  
$u^z$ by $v$,
denoted each 
$u^{z^{w(m,n)}}$ by $v\cdot w(p,q)$, and then added the new relations 
$u^{z^{w(m,n)}}\!\!=\,v\cdot w(p,q)$ in order \textit{to mimic} the isomorphism $\theta$. 

\smallskip
Further, $\langle z,m,n,u\rangle$ admits an isomorphism sending $z,m,n,u$ to $z^m\!,\, m,n,u$, and $\langle p,q,v \rangle$ admits an isomorphism sending $p,q,v$ to $p,q,v\cdot p$. These two isomorphisms agree on words of type $u^{z^{w(m,n)}}$ and $v\cdot w(p,q)$, and so they have a common continuation $\eta_1$ on the whole $\mathscr{P}$. We can similarly define an isomorphism $\eta_2$ on $\mathscr{P}$ sending 
 $z,m,n,u$ to $z^n\!,\, m,n,u$
 and 
$p,q,v$ to $p,q,v\cdot q$.
Using these isomorphisms define the finitely generated HNN-extensions:
$$
\bar{\mathscr{R}} = \bar{\mathscr{P}} *_{\eta_1,\eta_2} (\kappa_1, \kappa_2)
,\quad\quad\quad
\mathscr{R} = \mathscr{P} *_{\eta_1,\eta_2} (\kappa_1, \kappa_2)
$$
with $\mathscr{R} \le \bar{\mathscr{R}}$.
Then $\bar{\mathscr{R}}$ can be given by the relations:
\begin{enumerate}
\item the \textit{finitely} many relations of $K_{Z_{\mathcal T}}$;

\item the \textit{infinitely} many relations for $\theta$, i.e., those stating  $u^{z^{w(m,n)}}\!\!=v\cdot w(p,q)$ for all $w$;

\item the \textit{finitely} many relations telling the images of $7$ generators $z,m,n,u;p,q,v$  under isomorphisms $\eta_1$ and $\eta_2$.
\end{enumerate}

It is easy to see that the relations of point (2) are mainly redundant, and they can  be replaced by a \textit{single} relation $u^z=v$ which is $u^{z^{w(m,n)}}\!\!=v\cdot w(p,q)$ for the trivial word $w(m,n)=1$. The routine of verification is simple, and we display the idea just by an example. For, say, the word $w(m,n) = w_f(m,n) = m^3 n^5 m^4 n^7$ from point~\ref{SU If AB is benign then ZB also is benign} we from $u^z=v$ deduce:
\begin{equation}
\label{EQ 3 5 4 7}
u^{z^{m^3 n^5 m^4 n^7}}\!\! = v\cdot p^3 q^5 p^4 q^7
\end{equation}
in the following way. 
Putting the exponents in the above word $w$ in \textit{reverse order} write down the word $w'=w'(\kappa_1, \kappa_2) 
= \kappa_2^7 \kappa_1^4 \kappa_2^5 \kappa_1^3$ in stable letters $\kappa_1, \kappa_2$. Then:
\begin{equation}
\label{EQ kappa_2 kappa_1 first}
\begin{split}
\left(u^{z}\right)^{w'} \!&  
= \left(u^{z}\right)^{\kappa_2^{\boldsymbol 7} \kappa_1^4 \kappa_2^5 \kappa_1^3}
= \left((u^{\kappa_2})^{z^{\kappa_2}}\right)^{\kappa_2^{\boldsymbol 6} \kappa_1^4 \kappa_2^5 \kappa_1^3} 
= 
\left(\eta_2(u)^{\eta_2(z)}\right)^{\kappa_2^{\boldsymbol 6} \kappa_1^4 \kappa_2^5 \kappa_1^3}
= 
\left(u^{z^n}\right)^{\kappa_2^{\boldsymbol 6} \kappa_1^4 \kappa_2^5 \kappa_1^3}
\\
& 
= \big(u^{z^{n^7}}\big)^{\kappa_1^4 \kappa_2^5 \kappa_1^3}
= \big(u^{z^{m^4 n^7}}\big)^{\kappa_2^5 \kappa_1^3}
= \big(u^{z^{n^5 m^4 n^7}}\big)^{\kappa_1^3}
= u^{z^{m^3 n^5 m^4 n^7}};
\end{split}
\end{equation}
\begin{equation}
\label{EQ kappa_2 kappa_1 second}
\begin{split}
\hskip-1mm \left(v\right)^{w'} \!&  
= \left(v\right)^{\kappa_2^{\boldsymbol 7} \kappa_1^4 \kappa_2^5 \kappa_1^3}
= \left(v^{\kappa_2}\right)^{\kappa_2^{\boldsymbol 6} \kappa_1^4 \kappa_2^5 \kappa_1^3} 
= 
\left(\eta_2(v)\right)^{\kappa_2^{\boldsymbol 6} \kappa_1^4 \kappa_2^5 \kappa_1^3}
= 
\left(v \cdot q\right)^{\kappa_2^{\boldsymbol 6} \kappa_1^4 \kappa_2^5 \kappa_1^3}
\\
& 
= \left(v \cdot q^7\right)^{\kappa_1^4 \kappa_2^5 \kappa_1^3}
= \left(v \cdot p^4 q^7\right)^{\kappa_2^5 \kappa_1^3}
= \left(v \cdot q^5 p^4 q^7\right)^{\kappa_1^3}
= v \cdot p^3 q^5 p^4 q^7.
\end{split}
\end{equation}
So from $u^z = v$ and equalities \eqref{EQ kappa_2 kappa_1 first} and \eqref{EQ kappa_2 kappa_1 second} follows \eqref{EQ 3 5 4 7}, that is, $\bar{\mathscr{R}}$ is finitely presented.

We see that 
$Z_{\mathcal T}= \big\langle z^{w_f(m,n)} \mathrel{|} f \in \mathcal T \big\rangle$ 
is benign also in  $\mathscr{P}
=
\langle z,m,n,u\rangle 
*_{\theta}
\langle p,q,v \rangle$ for the finitely presented group $\bar{\mathscr{R}}$, and for its finitely generated subgroup $L_{Z_{\mathcal T}}$ (we do not change this subgroup after the previous step).

The finitely generated subgroup $\langle u,v \rangle$ trivially is benign in 
$\mathscr{P}$ for the finitely presented $\bar{\mathscr{R}}$, and for its finitely generated subgroup $\langle u,v \rangle$.
Hence by Corollary~\ref{CO intersection and join are benign multi-dimensional}\;\eqref{PO 2 CO intersection and join are benign multi-dimensional} the join $Z^+_{\mathcal T} = \big\langle z^{w_f(m,n)};\, u, v \mathrel{|} f \in \mathcal T \big\rangle$ is benign in $\mathscr{P}$ for the finitely presented $\bigast$-construction:
$$
K_{Z^+_{\mathcal T}}
=
\big(\bar{\mathscr{R}} *_{L_{Z_{\mathcal T}}} e_1\big) *_{\bar{\mathscr{R}}}
\big(\bar{\mathscr{R}} *_{\langle u,v \rangle} e_2\big),
$$
and for its finitely generated subgroup
$
L_{Z^+_{\mathcal T}}
=
\big\langle 
\mathscr{P}^{e_1}\!,\;
\mathscr{P}^{e_2}\!
\big\rangle
$.

Further, the finitely generated subgroup $\langle p,q \rangle$ is benign in 
$\mathscr{P}$ for the finitely presented $\bar{\mathscr{R}}$, and for its finitely generated subgroup $\langle p,q \rangle$.
Hence, if we also establish the equality:
\begin{equation}
\label{EQ important equality here}
\big\langle z^{w_f(m,n)};\, u, v \mathrel{|} f \in \mathcal T \big\rangle
\cap 
\langle p,q \rangle
= Z^+_{\mathcal T}
\cap 
\langle p,q \rangle
= W_{\mathcal T}  = 
\big\langle w_f(p,q) \mathrel{|} f \in \mathcal T\, \big\rangle,
\end{equation}
then by Corollary~\ref{CO intersection and join are benign multi-dimensional}\;\eqref{PO 1 CO intersection and join are benign multi-dimensional} the desired intersection $W_{\mathcal T}$ of the above two benign subgroups $Z^+_{\mathcal T}$ and $\langle p,q \rangle$ is benign in $\mathscr{P}$ for the finitely presented $\bigast$-construction:
$$
K_{W_{\mathcal T}}
=
\big(K_{Z^+_{\mathcal T}} *_{L^+_{Z_{\mathcal T}}} f_1\big) *_{\bar{\mathscr{R}}}
\big(\bar{\mathscr{R}} *_{\langle p,q \rangle} f_2\big),
$$
and for its finitely generated subgroup
$
L_{W_{\mathcal T}}
=
\mathscr{P}^{f_1 f_2}
$.
But since $W_{\mathcal T}$ entirely is inside $\langle p,q \rangle$, then from 
$\mathscr{P} \cap L_{W_{\mathcal T}} = W_{\mathcal T}$ it follows $\langle p,q \rangle \cap L_{W_{\mathcal T}} = W_{\mathcal T}$, that is, $W_{\mathcal T}$ is benign in $\langle p,q \rangle$ for the same 
$K_{W_{\mathcal T}}$ and $L_{W_{\mathcal T}}$.
\,
To fully finish the proof it remains to argument the equality \eqref{EQ important equality here}.
For every $f$ and the respective word $w(p,q) = w_f(p,q)$ we have:
$$
w(p,q)= v^{-1} \cdot  v\cdot w(p,q)
= v^{-1} \cdot u^{z^{w(m,n)}} 
\in Z^+_{\mathcal T}.
$$
And to see that the left-hand side of \eqref{EQ important equality here} is in $W_{\mathcal T}$ it is enough to apply the conjugate collection process of Point \ref{SU The conjugates collecting process} for $\X=\{u,v \}
$ and for 
$\Y=\{w_f(p,q) \mathrel{|} f \in \mathcal T \}$.

\subsection{Writing $K_{W_{\mathcal T}}$  by generators and defining relations}
\label{SU Writing KWT by generators and defining relations} 

From definition of $\bar{\mathscr{P}}$ in point~\ref{SU If ZT is benign then W(p,q) also is benign} (including \textbf{infinitely} many relations $u^{z^{w(m,n)}} = v\cdot w(p,q)$) we have:
\begin{equation}
\label{EQ relations bar P}
\begin{split}
\bar{\mathscr{P}} 
& = \big\langle
\text{$84$ generators of $K_{Z_{\mathcal T}}$\! from \eqref{EQ relations KZT}};\;\;  u, p, q, v
\mathrel{\,|\,} 
\text{$836$ relations of $K_{Z_{\mathcal T}}$\! from \eqref{EQ relations KZT}};\\
& \hskip11mm 
\text{\textbf{infinitely} many relations $u^{z^{w(m,n)}}\!\! = v\cdot w(p,q)$} 
\text{ for \textbf{all} $w(m,n)\in \langle m,n \rangle$}
\big\rangle.
\end{split}
\end{equation}
$\bar{\mathscr{P}}$ has $84+4=88$ generators and infinitely many relations.

Next we reduce the infinitely many relations to finitely many relations. Namely, from definition of $\bar{\mathscr{R}}$, $\eta_1$, $\eta_2$, $\kappa_1$, $\kappa_2$ in point~\ref{SU If ZT is benign then W(p,q) also is benign} we get: 
\begin{equation}
\label{EQ relations bar R}
\begin{split}
\bar{\mathscr{R}} 
& = \big\langle
\text{$84$ generators of $K_{Z_{\mathcal T}}$ from \eqref{EQ relations KZT}};\;\;\;  u, p, q, v;\;\;\;  \kappa_1, \kappa_2
\mathrel{\;\;|\;\;}\\
& \hskip11mm 
\text{$836$ relations of $K_{Z_{\mathcal T}}$ from \eqref{EQ relations KZT}}; \;\;\;
\text{\textbf{a single}  relation $u^z=v$};\\
& \hskip11mm 
\text{$\kappa_1$ sends 
$z,m,n,u; p,q,v$ to $z^m\!,\, m,n,u, p,q,v\cdot p$};\\[-3pt]
& \hskip11mm 
\text{$\kappa_2$ sends 
$z,m,n,u; p,q,v$ to $z^n\!,\, m,n,u, p,q,v\cdot q$}
\big\rangle.
\end{split}
\end{equation}
$\bar{\mathscr{R}}$ has $84+4+2=90$ generators and $836+1+14+14=865$ relations.

Hence using definition of $K_{Z^+_{\mathcal T}}$ and the $14$-generator group $L_{Z_{\mathcal T}}$ in point~\ref{SU If ZT is benign then W(p,q) also is benign}:
\begin{equation}
\label{EQ relations KZT+}
\begin{split}
K_{Z^+_{\mathcal T}}
& = \big\langle
\text{$90$ generators of $\bar{\mathscr{R}}$ from \eqref{EQ relations bar R}};\;  e_1,e_2
\mathrel{\;|\;} 
\text{$865$ relations of $\bar{\mathscr{R}}$ from \eqref{EQ relations bar R}};\\
& \hskip11mm 
\text{$e_1$ fixes conjugates of $a,b,c;\;  z, m, n, \kappa$ by $q_5 q_6$}; \;\;\;
\text{$e_2$ fixes 
$u,v$}
\big\rangle.
\end{split}
\end{equation}
$K_{Z^+_{\mathcal T}}$ has $90+2=92$ generators and $865+7+2=874$ relations, and $
L_{Z^+_{\mathcal T}}
=
\big\langle 
\mathscr{P}^{e_1}\!,\;
\mathscr{P}^{e_2}\!
\big\rangle
$ is a $14$-generator group.

Finally, using the definition of $K_{W_{\mathcal T}}$, $L_{Z^+_{\mathcal T}}$, $\bar{\mathscr{R}}$ again from point~\ref{SU If ZT is benign then W(p,q) also is benign}, we arrive to:
\begin{equation}
\label{EQ relations KWT}
\begin{split}
K_{W_{\mathcal T}} \!
& = \big\langle
\text{$92$ generators of $K_{Z^+_{\mathcal T}}$ from \eqref{EQ relations KZT+}};\;  f_1,f_2
\mathrel{\;\;|\;\;} \\
& \hskip11mm 
\text{$874$ relations of $K_{Z^+_{\mathcal T}}$ from \eqref{EQ relations KZT+}};\\
& \hskip11mm 
\text{$f_1$ fixes conjugates of $a,b,c;\;  z, m, n, \kappa$ by $e_1$ and by $e_2$}; 
\;\; \text{$f_2$ fixes 
$p,q$}
\big\rangle.
\end{split}
\end{equation}
$K_{W_{\mathcal T}}$ has $92+2=94$ generators and $874+7+7+2=890$ relations.
And its finitely generated subgroup
$L_{W_{\mathcal T}}
= \mathscr{P}^{f_1 f_2}$ is a $7$-generator group.

\subsection{The ``Higman Rope Trick'' and the final construction of $\mathcal{Q}$}
\label{SU The Higman Rope Trick}

In point~\ref{SU Embedding Q into a 2-generator group} we embedded 
$\Q$ into a $2$-generator overgroup $T_\Q$ given in \eqref{EQ TG genetic code} via the defining relations $w_k = w_k(x,y)$ of \eqref{EQ relations for Q detailed}. In point~\ref{SU Getting the  Higman code from TQ} we wrote down the sequences $f_k$ of \eqref{EQ Higman code for Q}, and denoted their set by $\mathcal T$.
We proved that $A_{\mathcal T}$ generated by all $a_{f_k}$ is benign in 
$\langle a,b,c \rangle$, and the respective $W_{\mathcal T} = W_{\mathcal T}(p,q)$ generated by all words $w_k(p,q)$ is benign in $\langle p,q \rangle$. Notice that the word $w_k(p,q)$ is obtained from $w_k(x,y)$ by just replacing $x,y$ by $p,q$.
Specifically, let $W_{\mathcal T}(p,q)$ be benign in $\langle p,q \rangle$ for the finitely presented group $K=K_{W_{\mathcal T}}$ and for its finitely generated subgroup  $L=L_{W_{\mathcal T}}$ computed in point~\ref{SU If ZT is benign then W(p,q) also is benign}.

In $F_2 = \langle x,y \rangle$ denoting $\mathscr{Y}=\big\langle W_{\mathcal T}(x,y)\big\rangle^{F_2}$ we have $T_\Q=F_2/\mathscr{Y}$. Let us show that this factor group can be explicitly embedded into a finitely presented group.

Fix a new stable letter $t$, and build the (finitely presented) $K_{W_{\mathcal T}} \!*_{L_{W_{\mathcal T}}}\! t$, inside which the subgroup $K_{W_{\mathcal T}}$ and its conjugate $K_{W_{\mathcal T}}^t$ generate their free product with amalgamation
$
K_{W_{\mathcal T}} \!*_{L_{W_{\mathcal T}}}\!
K_{W_{\mathcal T}}^t
$.
Since $\langle p,q \rangle \cap L_{W_{\mathcal T}} = W_{\mathcal T} = W_{\mathcal T} (p,q)$, then by \cite[Corollary~3.2\;(1)]{Auxiliary free constructions for explicit embeddings}
the subgroup $\langle p,q \rangle$  and its conjugate $\langle p,q \rangle^t$ generate their free product with amalgamation: 
\begin{equation}
\label{EQ small free product with amalgamation}
\langle p,q \rangle *_{W_{\mathcal T}}
\langle p,q \rangle^t
\end{equation}

In the direct product:
$$
\big( K_{W_{\mathcal T}} \!*_{L_{W_{\mathcal T}}}\! t \big) \times T_\Q
=
\big( K_{W_{\mathcal T}} \!*_{L_{W_{\mathcal T}}}\! t\big) \times 
F_2/\mathscr{Y} 
$$
for \textit{every} word $v(p,q) \in \langle p,q \rangle$, not necessarily a relation of type $w_k(p,q)$, we have an isomorphism $\langle p,q \rangle \to F_2/\mathscr{Y}$ given by the rule:
$$
 v(p,q) = \big( v(p,q) ,\;\; 1\big) 
\to 
\big( v(p,q) ,\;\; \mathscr{Y} v(x,y) \big),
$$
and also the identical isomorphism $\langle p,q \rangle^t \to F_2/\mathscr{Y}$ given by the rule:
$$
 v(p,q)^t = \big( v(p,q)^t ,\;\; 1\big) 
\to 
\big( v(p,q)^t ,\;\; 1\big).
$$
For every word $v(p,q)\in W_{\mathcal T} = W_{\mathcal T} (p,q)$ (in particular, for \textit{each} of relations $v(p,q)=w_k(p,q)$), the element (the coset)
$\mathscr{Y} v(x,y)$ is trivial in $F_2/\mathscr{Y}=T_\Q$. Hence the above two isomorphisms agree on the amalgamated subgroup $W_{\mathcal T}$ in \eqref{EQ small free product with amalgamation} (notice that $v(p,q)=v(p,q)^t$ for all words  $v(p,q) \in W_{\mathcal T}$).
Hence these two isomorphisms have a common continuation $\sigma$ on the whole \eqref{EQ small free product with amalgamation} using which we can define the HNN-extension:
\begin{equation}
\label{EQ definitin of X}
\mathcal{Q}
=
\Big(\big( K_{W_{\mathcal T}} \!*_{L_{W_{\mathcal T}}} \!t \big) \times T_\Q \Big) *_\sigma s.
\end{equation}
This group $\mathcal{Q}$ contains $T_\Q$, and for later use denote this identical embedding by:
\begin{equation}
\label{EQ define beta}
\beta: T_\Q \to \mathcal{Q}.
\end{equation} 
$\mathcal{Q}$ clearly is finitely generated, and so the desired embedding of $\Q$ into a \textit{finitely presented} group will be achieved, if we show that $\mathcal{Q}$ can be given by finitely many relations.

Let us list the relations following from the definition in \eqref{EQ definitin of X}:
\begin{enumerate}
\item the \textit{finitely} many relations of $K_{W_{\mathcal T}}$;

\item the \textit{finitely} many relations stating that $t$ fixes the \textit{finitely} many generators of $L_{W_{\mathcal T}}$;

\item the \textit{finitely} many relations stating that both generators $x,y$ of $T_\Q$ commute with $t$ and with the \textit{finitely} many generators of $K_{W_{\mathcal T}}$;

\item the \textit{infinitely} many relations $w_k(x,y)$ for $T_\Q$.

\item the \textit{finitely} many relations stating the action of $s$ on all words in $\langle p,q \rangle$ and $\langle p,q \rangle^t$: only \textit{four} relations are actually needed:
$p^s\!=\! p\mathscr{Y} x$, \,
$q^s\!=\!q\mathscr{Y} y$, \,
$(p^t)^s\!=\! p^t$, \,
$(q^t)^s\!=\!q^t$.
\end{enumerate}
If we now are in able to show that the relations $w_k(x,y)$ of the point (4) above are redundant, then $\mathcal{Q}$ will turn out to be \textit{a finitely presented group}.

For every $w_k(x,y)$ the respective $w_k(p,q)$ is in $\langle p,q \rangle$, and $w_k(p,q)^t$ is in $\langle p,q \rangle^t$. Hence, from four relations of the point (5) above it follows that: 
$$
w_k(x,y)^s = \big( w_k(p,q) ,\;\; \mathscr{Y} w_k(x,y) \big),
\quad \quad  
\big(w_k(x,y)^t\big)^s = \big( w_k(p,q)^t ,\;\; \mathscr{Y} \big).
$$
But since $w_k(x,y)\in  W_{\mathcal T} (p,q)$, we have $w_k(x,y)= w_k(x,y)^t$,\; 
$w_k(p,q)= w_k(p,q)^t$ and hence also
$\mathscr{Y} w_k(x,y) = \mathscr{Y}$, that is, $w_k(x,y)$ indeed is a relation of $T_\Q$. 

\medskip    
Hence the group $\mathcal{Q}$ is finitely presented, and as the desired embedding of the group $\Q$ into a finitely presented group for Theorem~\ref{TH embedding of Q into GP group} we may take the composition:
\begin{equation}
\label{EQ define varphi}
\varphi: \Q \to \mathcal{Q}
\end{equation} 
of the embedding $\alpha: \Q \to  T_\Q$ from \eqref{EQ define alpha} with the embedding $\beta: T_\Q \to  \mathcal{Q}$ from \eqref{EQ define beta}.
Thus, sections \ref{SU Writing Higman code for the embedding}--\ref{SE The final embedding} are an example of application of the $H$-machine for the group $\Q$.

\subsection{Writing $\mathcal{Q}$ by generators and defining relations}
\label{SU Writing Q by generators and defining relations} 

From definition of $\mathcal{Q}$, $\sigma$, $t$, $s$ in point~\ref{SU The Higman Rope Trick}, and definition of $L_{W_{\mathcal T}}
=
\mathscr{P}^{f_1 f_2}$ 
in point~\ref{SU If ZT is benign then W(p,q) also is benign} we have:
\begin{equation}
\label{EQ relations Q}
\begin{split}
\mathcal{Q}
& = \big\langle
\text{$94$ generators of $K_{W_{\mathcal T}}$ from \eqref{EQ relations KWT}};\;  x,y,\, t,s
\mathrel{\;\;|\;\;} \\
& \hskip11mm 
\text{$890$ relations of $K_{W_{\mathcal T}}$ from \eqref{EQ relations KWT}};\\
& \hskip11mm 
\text{$t$ fixes conjugates of $a,b,c;\;  z, m, n, \kappa$ by $f_1 f_2$}; \\
& \hskip11mm \text{$x,y$ commute with $t$ and with $94$ generators of $K_{W_{\mathcal T}}$ from \eqref{EQ relations KWT}}; \\
& \hskip11mm  
\text{$s$ sends $p, q, p^t\!,\, q^t$ to $px,\, qy,\, p^t\!\!,\; q^t$}
\big\rangle.
\end{split}
\end{equation}
The final group $\mathcal{Q}$ has $94+4=98$ generators and $890+7+ 2\cdot 94 +2+ 4=1091$ relations. We will list them all in the next section.

\section{Explicit list of generators and relations for $\mathcal{Q}$ and $T_{\!\mathcal{Q}}$}
\label{SE Explicit list of generators and relations for Q for TQ}

\subsection{Writing the explicit  list of generators and relations for $\mathcal{Q}$}
\label{SU Writing the explicit  list of generators and relations for Q} 

Collecting data prepared in points \ref{SU Writing K_B by generators and defining relations}, 
\ref{SU Writing K_omega B by generators and defining relations}, 
\ref{SU Writing K sigma - and K_C by generators and defining relations}, 
\ref{SU Writing K_F and K_T by generators and defining relations},  
\ref{SU Writing KZT by generators and defining relations}, 
\ref{SU Writing KWT by generators and defining relations}, 
\ref{SU Writing Q by generators and defining relations} we explicitly  list sets of generators and of defining relations for the finitely presented overgroup $\mathcal{Q}$ of $\Q$. Namely, as \textit{generators} of $\mathcal{Q}$ take:

\noindent
$b, c, t_2, t'_2, t_0, t'_0, r_1, r_2,
g,h,k,\, l_{0,0},\, l_{1,0},\, l_{1,1},
s_0, s_1, s_2,\, a, r,t_1, t'_1, u_1, u_2, d,e, v_1, v_2, q_1, q_2, q_3, q_4,$ \\
$\bar b, \bar c, \bar t_2, \bar t'_2, \bar t_0, \bar t'_0, \bar r_1, \bar r_2,
\bar g,\bar h,\bar k,\, \bar l_{0,0},\, \bar l_{1,0},\, \bar l_{1,1},
\bar s_0, \bar s_1, \bar s_2,\, \bar a, \bar r, 
\bar t_1, \bar t'_1, \bar u_1, \bar u_2,  \bar d,\bar e, \bar v_1, \bar v_2,\;\; \bar q_1, \bar q_2, \bar q_3, \bar q_4, $\\ 
$w_1, w_2, w_3, w_4, w_5, w_6,x_1, x_2,y_1, y_2,z_1, z_2,z, m, n, \kappa,  q_1, q_2, q_3, q_4,q_5, q_6,  u, p, q, v,  \kappa_1, \kappa_2,e_1,e_2,$\\ 
$f_1,f_2, \; x,y, t,s$
\; ($98$ generators in total).

\smallskip

And as \textit{defining relations} for $\mathcal{Q}$ take:

\noindent
$b^{t_2}= b^{c^{-1}}\!\!,\;
b^{t'_2}= b^{c^{-2}}\!\!,\;
b^{t_0}= b^{c}\!,\;
b^{t'_0}= b,\; c^{t_2}=c^{t'_2}=c^{t_0}=c^{t'_0}=c^2$;\;\; 
\\
$r_1$ fixes $b_2, t_2, t'_2$;\;\;
$r_2$ fixes $b_{-1}, t_0, t'_0$ (see \eqref{EQ relations Z});

\noindent
each of $g,h,k$ fixes 
$b^{r_1}\!,\; c^{r_1}\!,\; b^{r_2}\!,\; c^{r_2}$;\;\;
$l_{0,0}$ sends $
b, g, h$ to $
b, g^{h}\!, h$;\;\;
\\
$l_{1,0}$ sends $
b^c\!, g, h, h^k$ to $b^c\!, g^{h}\!, h, h^k$;\;\;
$l_{1,1}$ sends $
b^c\!,\; g, h, h^k$ to $
b^c\!,\; g^{h^c}\!\!\!,\; h^{h^k}\!\!\!,\; h^k$;\;\;
$s_0$ fixes $g$;\;\;
\\
$s_1$ fixes $g^{h} 
b^{-1} 
g^{-1}\!\!,\; l_{0, 0}$;\;\;
$s_2$ fixes $g^{h^k} 
b^{-c} 
g^{-1}\!\!,\;\; l_{1,0},\;\;l_{1,1}$  (see \eqref{EQ relations F});

\noindent
$a$ fixes conjugates of 
$b, c, t_2, t'_2, t_0, t'_0, r_1, r_2,
g,h,k$ by each of $s_0, s_1, s_2$;\;\;
\\
$r$ sends $a,b,c$ to $a,b^{c^2}\!\!\!\!,\;c$
(see \eqref{EQ relations D});

\noindent
$b^{t_1}= b,\;
b^{t'_1}= b^{c^{-1}}\!\!,\;
c^{t_1}=c^{t'_1}=c^2$;\;\;
$u_1$ fixes $b^c\!, t_1, t'_1$;\;\;
$u_2$ fixes $a, b, t_1, t'_1$;\;\;
\\
$d$ fixes $a^{u_1}\!\!, b^{u_1}\!\!, c^{u_1}$ and sends $a^{u_2}\!\!, b^{u_2}\!\!, c^{u_2}$ to $a^{b u_2}\!, b^{u_2}\!, c^{b u_2}$;\;\;
$e$ sends $a,b,c$ to $a,b^c\!,\ c$
(see \eqref{EQ relations A});

\noindent
$v_1$ fixes $a,d^e$;\;\;$v_2$ fixes $a^b\!,\, d d^e$
(see \eqref{EQ relations B});

\noindent
$q_1$ fixes 
$a^{v_1}\!,\, b^{v_1}\!,\, c^{v_1}\!,\, a^{v_2}\!,\, b^{v_2}\!,\, c^{v_2}$;\;\; 
$q_2$ fixes $b,c$
(see \eqref{EQ relations L});

\noindent
$q_3$ fixes the conjugates of 
$b, c, t_2, t'_2, t_0, t'_0, r_1, r_2,
g,h,k, l_{0,0},\, l_{1,0},\, l_{1,1},
s_0, s_1, s_2,\, a, r$ by each of $p_1,\, p_2$;\;\; 
$q_4$ fixes $a,b,c$
(see \eqref{EQ relations K omega2 B});

\noindent
copies of above relations (from 
\eqref{EQ relations Z}, \eqref{EQ relations F}, \eqref{EQ relations D}, \eqref{EQ relations A}, \eqref{EQ relations B}, \eqref{EQ relations L}, \eqref{EQ relations K omega2 B}) with \textit{bars} on each letter, i.e., 
$\bar b^{\bar t_2}= \bar b^{\bar c^{-1}}\!\!,\;
\bar b^{\bar t'_2}= \bar b^{\bar c^{-2}}\!\!,\;
\bar b^{\bar t_0}= \bar b^{\bar c}\!,\;
\bar b^{\bar t'_0}= \bar b
; \; \ldots \; ;$
$\bar q_4$ fixes $\bar a,\bar b,\bar c$
(see \eqref{EQ 152 relations with bars});

\noindent
$a,b,c$ commute with each of $\bar b, \bar c, \bar t_2, \bar t'_2, \bar t_0, \bar t'_0, \bar r_1, \bar r_2,
\bar g,\bar h,\bar k,\, \bar l_{0,0},\, \bar l_{1,0},\, \bar l_{1,1},
\bar s_0, \bar s_1, \bar s_2,\, \bar a, \bar r, \bar t_1, \bar t'_1, $ $\bar u_1, \bar u_2, \bar d,\bar e, \bar v_1, \bar v_2,\;\; \bar q_1, \bar q_2, \bar q_3, \bar q_4$
(see \eqref{EQ relations K omega2 B x G});

\noindent
$w_1$ fixes $\bar b, \bar c, \bar t_2, \bar t'_2, \bar t_0, \bar t'_0, \bar r_1, \bar r_2,
\bar g,\bar h,\bar k,\, \bar l_{0,0},\, \bar l_{1,0},\, \bar l_{1,1},
\bar s_0, \bar s_1, \bar s_2,\, \bar a, \bar r;\;\; a,b,c$;\;\;
\\
$w_2$ fixes $\bar a a$,
$\bar b  b^{c^{-1}}$\!\!\!,\;
$\bar c c$
(see \eqref{EQ relations M1});

\noindent
$w_3$ fixes $\bar a, \bar b, \bar c$;\;\;  $w_4$ fixes $\bar a^{w_1 w_2}, \bar b^{w_1 w_2}, \bar c^{w_1 w_2}, 
a^{w_1 w_2}, b^{w_1 w_2}, c^{w_1 w_2}$ 
(see \eqref{EQ relations M2});

\noindent
$w_5$ fixes $a,b,c$;\;\; $w_6$ fixes conjugates of $\bar a, \bar b, \bar c,\; a,b,c$ by $w_3$ and by $w_4$
(see \eqref{EQ relations M3});

\noindent
$x_1$ fixes conjugates of each of 
$b, c, t_2, t'_2, t_0, t'_0, r_1, r_2,
g,h,k,\, l_{0,0},\, l_{1,0},\, l_{1,1},$
\\
$s_0, s_1, s_2,\, a, r$
 \,by\, $p_3 p_4$; \;\;
$x_2$ fixes conjugates of $\bar a, \bar b, \bar c,\, a,b,c $ by $w_5 w_6$
(see \eqref{EQ relations KC});

\noindent
$y_1$ fixes $a^{x_1 x_2}\!\!,\;
 b^{x_1 x_2}\!,\; 
 c^{x_1 x_2}$;\;\; $y_2$ fixes $a^{(b^2)^{c^5}}\!\!\!,\; d^{e^5}\! d^{e^6}$
(see \eqref{EQ relations KF});

\noindent
$z_1$ fixes $a^{y_1 y_2}\!,\; b^{y_1 y_2}\!,\; c^{y_1 y_2}$ and  $d^{e^i}$ for $i=0, \ldots, 4,\; 6, \ldots, 18$; \\
$z_2$ fixes 
$a_{f_2}\!=a^{
b\,( b^{-3})^{c}\, (b^{-1})^{c^{2}}\! (b^{-3})^{c^{3}}\! (b^{-1})^{c^{4}}\! (b^{3})^{c^{5}}\; b^{c^{6}}\; (b^{3})^{c^{7}}\; b^{c^{8}}\! b^{c^{9}}\! (b^{-1})^{c^{10}}\! (b^{-2})^{c^{11}}\! (b^{-1})^{c^{12}}\! (b^{-1})^{c^{13}}\! b^{c^{14}}\!(b^{2})^{c^{15}}\! b^{c^{16}}\! (b^{2})^{c^{17}} \!(b^{-1})^{c^{18}}
}$\!\!\!,\;
and $z_2$ fixes  
$d'' \!= (d^{-1})^{e} (d^{-1})^{e^3} d^{e^5} d^{e^7} \! (d^{-1})^{e^{11}}\! d^{e^{15}} \!d^{e^{17}}$\!
(see \eqref{EQ relations KТ});

\noindent
$\kappa$ sends $z, m, n$ to $z, n, m$;\;\;\; 
each of $z, m, n, \kappa$ commutes with each of $b, c, t_2, t'_2,$  $t_0, t'_0, r_1, r_2, $ $
g,h,k,\, l_{0,0},\, l_{1,0},\, l_{1,1},
s_0, s_1, s_2, a, r,t_1, t'_1, u_1, u_2, d,e, v_1, v_2, q_1, q_2, q_3, q_4,\, \bar b, \bar c, \bar t_2, \bar t'_2, \bar t_0, \bar t'_0, \bar r_1, \bar r_2,$\\ 
$\bar g,\bar h,\bar k,\, \bar l_{0,0},\, \bar l_{1,0},\, \bar l_{1,1},
\bar s_0, \bar s_1, \bar s_2,\, \bar a, \bar r, 
\bar t_1, \bar t'_1, \bar u_1, \bar u_2,  \bar d,\bar e, \bar v_1, \bar v_2,\;\; \bar q_1, \bar q_2, \bar q_3, \bar q_4, 
w_1, w_2,$\\ 
$ w_3, w_4, w_5, w_6,x_1, x_2,y_1, y_2,z_1, z_2$; \;\;
$q_1$ fixes $a^{z_1 z_2}\!,\; b^{z_1 z_2}\!,\; c^{z_1 z_2}$ and $z, n, m, \kappa$
(see \eqref{EQ relations Q1});

\noindent
$q_3$ fixes $a,b,c$;\;\;\; $q_4$ fixes conjugates of $a,b,c;\;  z, m, n, \kappa$ by $q_1 q_2$
(see \eqref{EQ relations Q2});

\noindent
$q_5$ fixes $z, m, n, \kappa$;\;\; $q_6$ fixes the conjugates of $a,b,c;\;  z, m, n, \kappa$ by $q_3$ and by $q_4$
(see \eqref{EQ relations KZT});

\noindent
$u^z=v$;\;\;\;
$\kappa_1$ sends 
$z,m,n,u; p,q,v$ to $z^m\!,\, m,n,u, p,q,v\cdot p$;\\
$\kappa_2$ sends 
$z,m,n,u; p,q,v$ to $z^n\!,\, m,n,u, p,q,v\cdot q$
(see \eqref{EQ relations bar R});

\noindent
$e_1$ fixes conjugates of $a,b,c;\;  z, m, n, \kappa$ by $q_5 q_6$;\; 
$e_2$ fixes $u,v$
(see \eqref{EQ relations KZT+});

\noindent
$f_1$ fixes the conjugates of $a,b,c, z, m, n, \kappa$ by $e_1$ and by $e_2$;\;\;\; 
$f_2$ fixes $p,q$
(see \eqref{EQ relations KWT});

\noindent 
$t$ fixes conjugates of $a,b,c;\;  z, m, n, \kappa$ by $f_1 f_2$; \;\; 
$x,y$ commute with each of 
$b, c, t_2, t'_2, t_0, t'_0,$ 
$ r_1, r_2, g,h,k,\, l_{0,0},\, l_{1,0},\, l_{1,1},
s_0, s_1, s_2,\, a, r,t_1, t'_1, u_1, u_2, d,e, v_1, v_2, q_1, q_2, q_3, q_4, \bar b, \bar c, \bar t_2, \bar t'_2, \bar t_0, \bar t'_0, $ \\
$ \bar r_1, \bar r_2,
\bar g,\bar h,\bar k,\, \bar l_{0,0},\, \bar l_{1,0},\, \bar l_{1,1},
\bar s_0, \bar s_1, \bar s_2,\, \bar a, \bar r, 
\bar t_1, \bar t'_1, \bar u_1, \bar u_2,  \bar d,\bar e, \bar v_1, \bar v_2,\;\; \bar q_1, \bar q_2, \bar q_3, \bar q_4, 
w_1, w_2, w_3, w_4, $ \\
$w_5, w_6,x_1, x_2,y_1, y_2,z_1, z_2,z, m, n, \kappa,  q_1, q_2, q_3, q_4,q_5, q_6, u, p, q, v,  \kappa_1, \kappa_2,e_1,e_2,f_1,f_2,\; t$;
\\
$s$ sends $p, q, p^t\!,\, q^t$ to $px, qy, p^t\!,\, q^t$ 
(see \eqref{EQ relations Q})
\hfill
($1091$ relations in total).


\subsection{Writing the explicit  list of generators and relations for $T_{\!\mathcal{Q}}$}
\label{SU Writing the explicit  list of generators and relations for T cal Q} 
Using the method from \cite{Embeddings using universal words} already discussed in point~\ref{SU Embedding Q into a 2-generator group} we can reduce the number of generators of $\mathcal Q$ to just $2$, preserving the same number of relations though.
Namely, let $\langle \x, \y \rangle$ be a free group with two new generators $\x, \y $.
In point~\ref{SU Embedding Q into a 2-generator group} we used (3.6) and Theorem 3.2 from \cite{Embeddings using universal words} because $\Q$ is \textit{torsion-free}. But the group ${\mathcal Q}$ is {\it not} torsion-free, and hence we have to use just a little bit longer formula (1.1) with Theorem 1.1 from \cite{Embeddings using universal words}. 
Namely, for $98$ generators $b, c, \ldots ,s$ of ${\mathcal Q}$ listed in point~\ref{SU Writing the explicit  list of generators and relations for Q}, using (1.1) in \cite{Embeddings using universal words} we introduce $98$ words:
$$
a_i(\x, \y)=\y^{(\x \y^{\,k})^{\,2} \y^{-1}} \y^{- \x}
\in \langle x,y \rangle, \quad k=1,2, \ldots , 98,
$$
in $\langle \x, \y \rangle$,
and define the map $\gamma$ sending the $k$'th generator of ${\mathcal Q}$ to the $k$'th word $a_i(\x, \y)$:
\begin{equation}
\label{EQ define gamma}
\gamma: b\to a_1(\x, \y),\;\;\; \gamma: c\to a_2(\x, \y)\,,\;\ldots\;, \gamma: s\to a_{98}(\x, \y).
\end{equation} 
Then define $T_{\!\mathcal{Q}}$ by two generators $\x, \y$ and by defining relations obtained from the defining relations of $\mathcal{Q}$ by replacing each $k$'th generator by its image $a_k(\x, \y)$. Say, the first relation $b^{t_2}= b^{c^{-1}}$\!\! of $\mathcal{Q}$ (which can be rewritten as $b^{t_2}b^{-c^{-1}}$) turns to:
$$
r_1(\x, \y)\!=\!\big(a_1(\x, \y)\big)^{a_3(\x, \y)}
\big(a_1(\x, \y)\big)^{-(a_2(\x, \y))^{\,-1}}  
\!\!\!\!\!\!= \big(\y^{(\x \y)^{\,2} \y^{-1}} \!\! \y^{- \x} \big)^{\y^{(\x \y^{\,3})^{\,2} \y^{-1}}  \y^{- \x}}
\!
\big(\y^{(\x \y)^{\,2} \y^{-1}} \! \y^{- \x} \big)^{-(\y^{(\x \y^{\,2})^{\,2} \y^{-1}} \! \y^{- \x})^{-1}}
$$
($b, c$ and $t_2$ respectively are the $1$'st, $2$'nd and $3$'rd generators of $\mathcal{Q}$).
Changing this way all the relations of $\mathcal{Q}$ we  write the $2$-generator group containing $\mathcal{Q}$, and hence $\Q$ also:
\begin{equation}
\label{EQ T cal Q}
T_{\!\mathcal{Q}}=\big\langle 
\x , \y  \mathrel{|}
r_1(\x, \y),\ldots,r_{1091}(\x, \y)
\big\rangle.
\end{equation}
The explicit embedding $\psi:\Q \to T_{\!\mathcal{Q}}$ is the composition of $\alpha, \beta,\gamma$,\, or of $\varphi$ and $\gamma$.

\begin{Remark}
\label{RE It is possible to reduce the above number of relations}
At the cost of significant complication of the proofs it is possible to reduce the above number of relations $1091$. However, we feel our goal has been achieved, and any further complexity of the argument is not reasonable.
\end{Remark}

\vskip-2mm

\medskip
\noindent 
E-mail:
\href{mailto:v.mikaelian@gmail.com}{v.mikaelian@gmail.com}
$\vphantom{b^{b^{b^{b^b}}}}$

\noindent 
Web: 
\href{https://www.researchgate.net/profile/Vahagn-Mikaelian}{researchgate.net/profile/Vahagn-Mikaelian}

\end{document}